\documentclass[12pt,reqno]{amsart}
\usepackage{latexsym}
\usepackage{amssymb}

\usepackage{epsfig}

\newcommand{\PR}{{\bf Prob}}

\newcommand{\szego}{Szeg\"o\ }

\newcommand{\kahler}{K\"ahler\ }

\newcommand{\PP}{{\mathbb P}}
\newcommand{\R}{{\mathbb R}}
\newcommand{\C}{{\mathbb C}}

\newcommand{\Z}{{\mathbb Z}}

\newcommand{\CP}{\C\PP}

\newcommand{\dbar}{\bar\partial}
\newcommand{\ddbar}{\partial\dbar}

\newcommand{\E}{{\mathbf E}}

\renewcommand{\phi}{\varphi}

\newcommand{\acal}{\mathcal{A}}
\newcommand{\bcal}{\mathcal{B}}

\newcommand{\dcal}{\mathcal{D}}
\newcommand{\ecal}{\mathcal{E}}

\newcommand{\jcal}{\mathcal{J}}
\newcommand{\lcal}{\mathcal{L}}
\newcommand{\pcal}{\mathcal{P}}
\newcommand{\mcal}{\mathcal{M}}

\newcommand{\ocal}{\mathcal{O}}

\newcommand{\zcal}{\mathcal{Z}}

\newcommand{\ga}{\gamma}

\newcommand{\de}{\delta}

\newtheorem{theo}{{\sc Theorem}}
\newtheorem{cor}[theo]{{\sc Corollary}}

\newtheorem{lem}[theo]{{\sc Lemma}}
\newtheorem{prop}[theo]{{\sc Proposition}}
\newtheorem{maintheo}{{\sc Theorem}}

\newenvironment{rem}{\medskip\noindent{\it Remark:\/} }{\medskip}
\newenvironment{defin}{\medskip\noindent{\it Definition:\/} }{\medskip}

\title[Large deviations of empirical zero point measures on Riemann
surfaces, I: $g = 0$
 ] {Large deviations of empirical zero point measures on Riemann surfaces, I: $g = 0$}

\author{Ofer Zeitouni}
\address{Faculty of Mathematics, Weizmann Institute of Science, Rehovot 76100, Israel and School of Mathematics, University of Minnesota, Minneapolis 55455, USA}
\email{zeitouni@math.umn.edu}

\author{Steve Zelditch}
\address{Department of Mathematics, Johns Hopkins University, Baltimore,
MD 21218, USA} \email{szelditch@jhu.edu}

\thanks{Research partially supported by NSF grant
DMS-0804133
and by a grant from the Israel Science Foundation (first author), and
by NSF grant
  DMS-0603850 (second author).}

\date{April 26,  2009}

\begin{document}

\begin{abstract} We prove a large deviation principle for empirical measures $$ Z_s: = \frac{1}{N} \sum_{\zeta: s(\zeta) = 0}
 \delta_{\zeta}, \;\;\; (N: = \# \{\zeta: s(\zeta) = 0)\}$$
 of zeros of random polynomials  in one variable. By random polynomial, we mean a  Gaussian measure
 on the space $\pcal_N = H^0(\CP^1, \ocal(N))$ determined by inner products $G_N(h, \nu)$ induced
 by any smooth Hermitian metric $h$ on $\ocal(1) \to \CP^1$ and any probability measure $d\nu$ on $\CP^1$ satisfying the weighted
 Bernstein-Markov inequality. The
 speed of the LDP is $N^2$ and the  rate function is closely related to the weighted energy of probability measures on $\CP^1$,
 and in particular its unique minimizer is the weighted equilibrium measure.
 \end{abstract}

\maketitle

 \section{\label{introduction} Introduction and statement
 of results}
The purpose of this article is to establish a large deviations
principle for the empirical measure \begin{equation}\label{ZN}
Z_s: = d\mu_{\zeta}: =  \frac{1}{N} \sum_{\zeta: s(\zeta) = 0}
 \delta_{\zeta}, \;\;\; N: = \# \{\zeta: s(\zeta) = 0\} \end{equation}
  of zeros of a
random polynomial $s$ of degree $N$.  Here, $\delta_{\zeta}$ is
the Dirac point measure at $\zeta \in \C$.  We define random
polynomials of degree $N$ by putting geometrically defined
Gaussian probability measures $d\gamma_N$ on the space $\pcal_N$
of holomorphic polynomials of degree $N$, or equivalently,
Fubini-Study measures $dV^{FS}_N$ on the projective space $\PP
\pcal_N$ of polynomials (see \S \ref{BACKGROUND} for
background). The measures $d\gamma_{N}, dV^{FS}_N$ are
determined by  a pair $(h = e^{- \phi}, \nu)$ consisting of  a
`weight' $\phi$ or (globally) a
 Hermitian metric $h$ on the hyperplane line bundle $\ocal(1) \to \CP^1$,  and
a probability measure $\nu$  on $\CP^1$ satisfying the
Bernstein-Markov condition (\ref{BM}). The Gaussian measure on
$H^0(\CP^1, \ocal(N))$ and the Fubini-Study measure on $\PP
H^0(\CP^1, \ocal(N))$  are induced from the Hermitian inner
products
\begin{equation} \label{INNERPRODUCT} ||s||^2_{G_N(h, \nu)} :=
\int_C \|s(z)\|^2_{h^N} d\nu(z), \;\; (s \in \pcal_N).
\end{equation}
 The zeros then become  equidistributed with high probability  in the large $N$ limit according to an equilibrium measure $d\nu_{h, K}$
 depending on $h$ and the support $K$ of $\nu$,
which reflects the competition between the  repulsion of nearby
zeros and the force of the  external electric field (curvature
form) $\omega_h$ of $h$ (see \cite{SZ,Ber1,Ber2,BB}). The
   large deviations results show that the empirical measures (\ref{ZN}) are concentrated exponentially
   closely (with speed $N^2$) to $\nu_{h, K}$ as $N \to \infty$, with rate given by a rate function $\tilde 
I^{h,K}$ that is minimized by $\nu_{h, K}$.

The large deviations rate function is determined from the joint
probability density $D_N(\zeta_1, \dots, \zeta_{N})$
 of zeros, which
measures the likelihood of a given configuration of $N$ points
arising as zeros of $s \in \PP H^0(\CP^1, \ocal(N))$. The joint
probability density is the density of a joint probability current
  on  the configuration space
 \begin{equation} \label{CONFIG} (\CP^1)^{(N)} = Sym^{N} \CP^1 :=  \underbrace{\CP^1\times\cdots\times \CP^1}_N /S_{N} \end{equation}
 of $N$ points of  $\CP^1$.   Here, $S_N$ is the symmetric
  group on $N$ letters. The joint probability current is by definition the pushforward
  \begin{equation} \label{JPCDEF}\vec K_n^N(\zeta_1, \dots, \zeta_N) : =  \dcal_* dV_N^{FS} \end{equation}
  of the Fubini-Study measure on $\PP \pcal_N$ under the  `zero set' or divisor map
  \begin{equation} \label{ZEROSMAP} \dcal: \pcal_N \to (\CP^1)^{(N)}, \;\;\;\dcal(p)= \zeta_1 + \cdots + \zeta_N, \end{equation}
  where $\{\zeta_1, \dots, \zeta_N\}$ is the zero set of $s$.
 Following a   standard  notation in algebraic geometry, we are
 writing
  an unordered set of points $ \{\zeta_1,\dots, \zeta_{N}\} $   as a formal sum (i.e.  a divisor) $\zeta_1 + \cdots + \zeta_N \in (\CP^1)^{(N)}$.
In the case of polynomials, $\dcal$ is obviously surjective (any
$N$-tuple of points is the zero set of some polynomial of degree
$N$);   one may identify $\PP \pcal_N \simeq (\CP^1)^{(N)}$.

 The zero set can also be encoded by the probability
  measure  (\ref{ZN}) on $\CP^1$. This identification defines  a map
 \begin{equation}\label{DELTADEF} \mu : (\CP^1)^{(N)} \to \mcal(\CP^1), \;\;\; d\mu_{\zeta_1 +  \dots + \zeta_{N}} = d\mu_{\zeta}: =
  \frac{1}{N} \sum_{j = 1}^{N}
 \delta_{\zeta_j}, \end{equation}
 where $\mcal(\CP^1)$ is the (Polish) space
 of probability measures on $\CP^1$, equipped with
 the weak-* topology (i.e. the topology induced by
 weak, or equivalently vague, convergence of measures). In general, for any closed
 subset $F \subset \CP^1$ we denote by $\mcal(F)$ the probability
 measures supported on $F$.
  Thus, the zero sets can all be embedded as elements of
  the space $\mcal(\CP^1)$ of probability measures on $\CP^1$. This point of view is ideal for taking large $N$ limits, and has
  been previously used in many similar situations, for instance in analyzing the eigenvalues of random matrices \cite{BG,BZ,HP00a}.

Under  the  map $p \to \mu \circ \dcal(p)$ we further push forward
the joint probability current
 to obtain a probability  measure
\begin{equation} \label{LDPNa} \PR_N =  \mu_* \dcal_* dV_{FS}^N \end{equation}  on
 $\mcal(\CP^1)$.  Our main
results show that this  sequence of measures $\PR_N$ satisfies a
large deviations principle with speed $N^2$ and with a rate
function $I^{h,K}$ reflecting the choice of $(h, \nu)$.  Roughly
speaking this means that for any Borel subset $E \subset
\mcal(\CP^1)$,
$$\frac{1}{N^2} \log \PR_N \{\sigma \in \mcal: \sigma \in E\} 
\to - \inf_{\sigma \in E} I^{h,K}(\sigma). $$

Before stating our results, we  recall some notation and
background. Throughout this article we use the language of complex
geometry, and in particular we identity $\pcal_N = H^0(\CP^1,
\ocal(N))$, i.e. we identify polynomials of degree $N$ with
holomorphic sections of the $N$th power of the hyperplane line
bundle \cite{GH} (Chapter I.3). In the affine chart $U = \CP^1
\backslash \{\infty\}$, and in the standard holomorphic frame $e:
U \to \CP^1$, the Hermitian metric $h$ is represented by the
function $||e||_{h}^2 = e^{- \phi}$.

In weighted potential theory, the function $\phi$ is referred to
as a weight \cite{ST,B,B2}. Several authors have generalized
weighted potential theory to \kahler manifolds, and we will use
their geometric language \cite{GZ,Ber1,Ber2,B,B2,BS}. In the local
frame any holomorphic section may be written $s = f e$ where $f
\in \ocal(U)$ is a local holomorphic function. The inner product
(\ref{INNERPRODUCT}) then takes the form,
\begin{equation} \label{INNERPRODUCTa} ||s||_{G_N(h, \nu)} =
\int_{\C} |f(z)|^2 e^{-  N \phi} d\nu(z).
\end{equation}
The measure $\nu$ is assumed to satisfy  the weighted
Bernstein-Markov condition (see \cite{B2} (3.2) or \cite{BB},
Definition 4.3 and references):

\noindent
{\it For all $\epsilon
> 0$ there exists $C_{\epsilon} > 0$ so that
\begin{equation} \label{BM} \sup_K \|s(z)\|_{h^N} \leq C_{\epsilon}
    e^{ \epsilon N}
||s||_{G_N(h, \nu)}\,, \quad s\in H_0(\CP^1,\ocal(N)).
\end{equation}
} Here, and throughout this article, we write  \begin{equation}
\label{K} K = \;\; \mbox{ supp}\; \nu. \end{equation} We further
assume that \begin{equation} \label{REGULAR} K\;  {\it is\;
regular}
\end{equation} in the sense that every point of $\partial K$ is a regular
point. For this paper, the operative meaning of regular is that
for any $z^* \in
\partial K$, the capacity $\mbox{Cap}_h(D(z^*, \epsilon) \cap K) >
0$ for every $\epsilon > 0$, where $\mbox{Cap}_h$ is the Green's
capacity (see Definition \ref{CAPKOMEGA}). Here, $D(z^*,
\epsilon)$ is a metric disc of radius $\epsilon;$ the condition is
independent of the choice of metric. We refer to \S \ref{IRREG}
and to \cite{Lan} for background on regular and irregular points
of compact sets.

 We then define the
Gaussian probability measures $\gamma_{h^N, \nu}$ on $\pcal_N =
H^0(\CP^1, \ocal(N))$ as the Gaussian measure determined by the inner
product (\ref{INNERPRODUCT}) (the definition is reviewed in \S
\ref{GFS}).  The associated Fubini-Study measures are denoted by
$dV_{h^N, \nu}^{FS}$ on $\PP H^0(\CP^1, \ocal(N))$
(see \S \ref{FSDEF} for
the definition).

 In order to take advantage of the compactness of $\CP^1$, we now
introduce some notions of \kahler potential theory which will be
discussed further in \S \ref{GREEN}. Let $\omega_h$ be the
curvature $(1,1)$ form of a smooth Hermitian metric $h$ on
$\CP^1$.
 The Green's
function $G_h$ relative to $\omega_h$ is defined to be  the unique
solution $G_h(z, \cdot) \in \dcal'(\CP^1)$ of
\begin{equation} \label{GHDEF} \left\{ \begin{array}{ll} (i) &  dd^c_w G_h(z,w)   =  \delta_z(w) -
\omega_h (w), \\ \\ (ii) & G_h(z,w) = G_h(w, z), \\ & \\(iii) &
\int_{\CP^1} G_h(z,w) \omega_h(w) = 0, \end{array} \right.
\end{equation}
where the equality in the top line is in the sense of $(1,1)$
forms. Existence of $G_h$ is guaranteed by the $\ddbar$ Lemma even
when $\omega_h$ is non-positive, i.e. is not a \kahler form;
uniqueness follows from  condition (iii).
 As shown in Lemma \ref{GHFORM} of \S \ref{GREEN}, in the frame $e(z)$ over the affine chart $\C$ in which
 $h = e^{- \phi}$ and $\omega_h = dd^c \phi$, the Green's function has the local expression,
\begin{equation} \label{GHFORMintro}  G_{h}(z,w) = 2 \log
|z - w| - \phi(z) - \phi(w) + E(h),
\end{equation}
where \begin{equation} \label{EH} E(h): =  \left( \int_{\CP^1}
\phi(z) \omega_h + \rho_{\phi}(\infty) \right), \end{equation}
 $\rho_{\phi}$ being  a certain Robin constant (see (\ref{ROBINDEF}) of \S
 \ref{IP}, and also (\ref{LOGROBIN}) and (\ref{INTEGRALa})).

The Green's potential of a measure $\mu$ (with respect to
$\omega_h$) is defined by
\begin{equation}\label{Green's potential}  U^{\mu}_{h} (z) = \int_{\CP^1}
G_{h}(z,w) d\mu(w), \end{equation} and the Green's energy by
\begin{equation}\label{GEN} \ecal_{h}(\mu) = \int_{\CP^1 \times \CP^1} G_{h}(z,w)
d\mu(z) d\mu(w). \end{equation}

We will see, c.f. Lemma \ref{lem-ue} and Proposition \ref{LSCJG},
that under Assumptions (\ref{BM}) and (\ref{REGULAR}), for any
$\mu \in \mcal(\CP^1)$, $\ecal_{h}(\mu)<\infty$ and $|\sup_K
U_h^\mu|<\infty$. In particular, the function
\begin{equation} \label{IGREEN}  I^{h, K} (\mu) =
- \frac{1}{2} \ecal_{h}(\mu) + \sup_K U^{\mu}_{h},\quad \mu\in
\mcal(\CP^1)
\end{equation}
is well-defined (with $+\infty$ as possible value).
Set
\begin{equation}
    \label{eq-ofer1}
    E_0(h)=\inf_{\mu\in \mcal(\CP^1)} I^{h,K}(\mu), \quad
    \tilde I^{h,K}=I^{h,K}-E_0(h)\,.
\end{equation}
The infimum $\inf_{\mu\in \mcal(\CP^1)} I^{h,K}(\mu)$ is achieved
at the Green's equilibrium measure $\nu_{h, K}$ with respect to
$(h, K)$, and   $E_0(h)= \frac12 \log \mbox{Cap}_{\omega_h} (K)$,
where $\mbox{Cap}_{\omega_h}(K)$ is the Green's capacity. See
Lemma \ref{NC} (proved in \S \ref{SLATERG0}). By the Green's
equilibrium measure we mean the minimizer of $- \ecal_h$ on
$\mcal(K)$. We refer to   \S \ref{WEM} for definitions and
discussion  of $\nu_{h, K}$ and of $\mbox{Cap}_{\omega_h} (K)$
(See \eqref{CAPKOMEGA}).

\subsection{Statement of results}
Our main result is the following:
\begin{maintheo}\label{g0}  Let $h$ be a smooth Hermitian metric
 on $\ocal(1) \to \CP^1$ and let
$d\nu \in  \mcal(\CP^1)$ satisfy the Bernstein-Markov property
(\ref{BM}) and the regularity assumption (\ref{REGULAR}). Then
$\tilde I^{h,K}$ of
\eqref{eq-ofer1} is a convex rate function and the sequence of
probability measures $\{{\bf Prob}_N\}$ on $\mcal(\CP^1)$ defined
by (\ref{LDPNa}) satisfies a large deviations principle with speed
$N^2$ and rate function
$\tilde I^{h, K} $. Further,
there exists a unique measure $\nu_{h, K} \in \mcal(\CP^1)$
minimizing $\tilde I^{h, K}$, namely  the Green's equilibrium measure of
$K$ with respect to $h$.
\end{maintheo}
(Recall that a function $I:\mcal(\CP^1)\to \R$ is a rate
function if it is lower semicontinuous and non-negative.)

 Theorem
\ref{g0}
shows that the empirical
measures $d\mu_{\zeta}$, see (\ref{ZN}), concentrate near $\nu_{h, K}$
at an exponential rate. More precisely, if $B(\sigma, \delta)$
denotes the ball of radius $\delta$ around $\sigma \in
\mcal(\CP^1)$ in the Wasserstein metric, and
$B^o(\sigma,\delta)$ (respectively, $\overline{B(\sigma,\delta)}$)
denote its interior (respectively, its closure),
then \begin{equation}
\label{LDPROP}
\begin{array}{lll} - \inf_{\mu \in B^o(\sigma, \delta)} \tilde I^{h, K}(\mu)
&\leq & \liminf_{N \to \infty}
\frac{1}{N^2} \log {\bf Prob} _N(B(\sigma, \delta)) \\ && \\
&\leq&  \limsup_{N \to \infty} \frac{1}{N^2} \log {\bf Prob}
_N(B(\sigma, \delta)) \leq    - \inf_{\mu \in \overline{B(\sigma,
\delta)}} \tilde I^{h, K} (\mu). \end{array}
\end{equation} This implies afortiori that the expected value of
$d\mu_{\zeta}$ tends to $\nu_{h, K}$, refining  the result of
\cite{SZ} on the equilibrium distribution of zeros in the
unweighted case and the more general results in the  subsequent
articles \cite{B,BS,Ber1,Ber2}. Intuitively,   in the unweighted
case, zeros repel each other like electrons to the outer boundary
of $K$. A Hermitian metric or weight $h = e^{- \phi}$ with
$\omega_{\phi} > 0$  behaves like an uphill potential which pushes
electrons back into the interior of $K$ and gives rise to an
equilibrium potential which charges the interior of $K$, with
extra accumulation along $\partial K$.

The inner product (\ref{INNERPRODUCT}) depends only on  the
restriction of the metric   $h$  to $K$, see (\ref{K}), and
consequently the rate function should only depend on this
restriction. To see this, we rewrite it  in the standard affine
chart $\C$ and frame for $\ocal(1)$ in the form
\begin{equation} \label{IHKLOCAL} - \frac{1}{2} \ecal_{h}(\mu) +
\sup_K U^{\mu}_{h} = - \Sigma (\mu) + \sup_{z \in K} \{ 2
\int_{\C} \log |z - w| d\mu(w) - \phi(z)\},
\end{equation}  where
\begin{equation} \label{SIGMADEF}  \Sigma(\mu) = \int_{\C \times \C} \log |z - w| d\mu(z) d\mu(w) \end{equation}
is the logarithmic energy or entropy function. In the large
deviations analysis, it is more convenient to use the  formula in
Theorem \ref{g0} which uses the `compactification' of the metric
to $\CP^1$.

\subsection{Examples}

As an illustration  of the methods and results, we observe that
Theorem \ref{g0} applies to the  Kac-Hammersley ensemble as in
\cite{SZ}, where $d\nu = \delta_{S^1}$ (the invariant probability
measure on the unit circle), and where the weight $e^{- \phi} =
1$. Hence, the inner product is simply $ \frac{1}{2 \pi} \int_0^{2
\pi} |f(e^{i \theta})|^2 d\theta. $ It is simple to verify that
$d\nu = \delta_{S^1}$ satisfies the Bernstein-Markov property,
i.e. that for holomorphic polynomials of degree $N$,
$||p_N||_{S^1} \leq C_{\epsilon} e^{\epsilon N} \frac{1}{2 \pi}
\left(\int |p_N(e^{i \theta})|^2 d\theta\right)^{1/2}. $ Indeed,
we let $\Pi_N(z,w) = \sum_0^N z^n \bar{w}^n$ denote the \szego
reproducing kernel for $\pcal_N$ with this measure. Then by the
Schwarz inequality,
$$\sup_{z \in S^1} |p_N(z)| \leq \sup_{z \in S^1}
\sqrt{\Pi_N(z,z)} ||p_N||_{L^2(\nu)} \leq \sqrt{N}
||p_N||_{L^2(\nu)}. $$

 On $S^1$ we are
taking the weight to be `flat', i.e. the Hermitian metric to be
$\equiv 1$. We  are free to choose a smooth extension of this
Hermitian metric  to $\ocal(1) \to \CP^1$. For instance, we may
take $h = e^{- \phi}$ to be $S^1$ invariant, equal $1$ in a
neighborhood of $\CP^1$ and to equal the Fubini-Study metric in a
neighborhood of $\infty$. There is of course no unique choice of
the smooth  extension.  With any of these extensions,
$\delta_{S^1}$ is easily seen to satisfy the condition
(\ref{REGULAR}).

At the opposite extreme, the methods and results apply to the case
where $d\nu = \omega_{FS}$, the Fubini-Study \kahler form, and
where $h = h_{FS} = e^{- \log (1 + |z|^2)}. $ The Bernstein-Markov
property follows from the same calculation as in the
Kac-Hammersley example, except that the \szego reproducing  kernel
 is different (but
still equals $N + 1$ on the diagonal; see \cite{SZ,SZ2,SZ3} for
further background). The regularity  condition is obviously
satisfied.

\subsection{An application - hole probabilities}

The large deviation results give an accurate upper bound for
`hole probabilities' for our ensembles of Gaussian random
polynomials of one complex variable. A hole probability for an
open set $U$ is the probability that the random polynomial has no
zeros in $U$. Large deviations estimates for hole probabilities
for balls $U = B_R$ of increasing radius were proved in
\cite{SoTs} for certain random analytic functions.  More in line
with the present paper are asymptotic  hole probabilities as the
degree $N \to \infty$ of random holomorphic sections of powers
$L^N \to M$ of positive line bundles in \cite{SZZr}. The results
there hold in all dimensions, but the stronger assumption is made
that  $h$ is a Hermitian metric with positive curvature $(1,1)$
form.

We now state a hole probability for our  general Gaussian
ensembles on $\CP^1$, where the hole is an open set $U \subset
\CP^1$. We consider the
$$A_U=\{\mu\in {\mcal}(\C): \mu( U)=0\}\,.$$ The following
hole probability has the same speed of  exponential decay as in
\cite{SZZr}.

\begin{cor}
    \label{cor-hole1}
    For any of the Gaussian ensembles $G_N(h, \nu)$ and
for any open set $U$,
\begin{equation}
\label{eq-of2a} \limsup_{N\to\infty}\frac{1}{N^2}\log \PR_N(A_U)
\leq -\inf_{\mu\in A_U} \tilde I^{h,K}(\mu)\,.
\end{equation}
\end{cor}

\begin{proof}

If $\mu_n\to \mu$ weakly in $\mcal(\C)$ then $\liminf_{n\to
\infty} \mu_n(U^c)\leq \mu(U^c)$. Thus, $A_U$ is a closed set,
both in $\mcal(\C)$ and (with a slight abuse of notation) in
$\mcal(\C\PP^1)$. The upper bound is then immediate from  Theorem
\ref{g0}.

\end{proof}

Unfortunately, the large deviation principle is not quite strong
enough to provide complementary lower bounds. Indeed, the set
$$A_U^o=\{\mu\in\mcal(\C): \mu(U)=1\}$$
has empty interior for any set $U\neq \CP^1$,
and the large deviations lower bound is $-
\infty$.  The best one can obtain from the LDP is that, with $U$ closed and
the
set
$$A_U^{p,o}=\{\mu\in\mcal(\C): \mu(U)>p\}\,,$$
one has by a similar analysis
\begin{equation}
\label{eq-of3} \lim_{p\nearrow 1}
\liminf_{N\to\infty}\frac{1}{N^2}\log \PR_N(A_U^{p,o}) \geq
-\inf_{\mu\in A_U^o} \tilde I^{h,K}(\mu)\,.
\end{equation}

The constrained infimum $\inf_{\mu\in A_U} \tilde I^{h,K}(\mu)$ is
achieved by a measure $\nu_{U, h, K}$, which may be regarded as
a relative weighted equilibrium measure with respect to the two
independent sets $U, K$. In general it is impossible to evaluate
numerically. In a special case, we can however evaluate it. With $r<1$,
let
$U^c = \bar B_r\subset \C$ be the closed ball of radius $r$ centered at
the origin.
Set
$$A_r=\{\mu\in {\mcal}(\C): \mu(\bar B_r)=1\}\,.$$

\begin{cor} For the Kac-Hammersley  ensemble, and for $r < 1$,  we have
\begin{equation}
\label{eq-of2} \limsup_{N\to\infty}\frac{1}{N^2}\log \PR_N(A_r)
\leq \log r\,.
\end{equation}
\end{cor}

\begin{proof}
    By Corollary \ref{cor-hole1},
\begin{equation}
\label{eq-of1}
\limsup_{N\to\infty}\frac{1}{N^2}\log \PR_N(A_r)
\leq -\inf_{\mu\in A_r} \tilde I^{h,K}(\mu)\,.
\end{equation}

We specialize the last expression in the case of the Kac-Hammersley
ensemble: written in
the affine chart around $0$, we have
$$ \tilde I^{h,K}(\mu)=-\Sigma(\mu)+2\sup_{z\in S^1}
\int_\C \log |z-w|d\mu(w)\,,$$
where we used that for any $R\geq 1$,
$$\inf_{\mu\in A_R}I^{h,K}(\mu)=
-\Sigma(\nu)+2\int_\C \log|1-w|d\nu(w)=0\,,$$
with $\nu=\delta_{S^1}$  the uniform distribution on $S^1$.

Fix $r<1$. For given $\mu\in A_r$, let $\tilde \mu$ denote the radial
 symmetrization of $\mu$, that is, for any
measurable
$A\subset \C$,
$$\tilde \mu(A)=\frac{1}{2\pi}\int_0^{2\pi}\int {\bf 1}_{ze^{i\theta} \in A}
d\mu(z) d\theta\,.$$
Due to the convexity of
$A_r$ and of $I^{h,K}(\cdot)$,
the minimizer $\mu^*$ in the right side of \eqref{eq-of1} is
radially symmetric, i.e. $\tilde \mu^*=\mu^*$,
and  supported in $A_r$.
Using the identity, valid
for any $s\leq 1$,
\begin{equation}
\label{patsats}
\int_0^{2\pi} \log  |1-se^{i\theta}|d\theta=0\,,
\end{equation}
we thus obtain
$$\inf_{\mu\in A_r} \tilde I^{h,K}(\mu)=
[\inf_{\mu\in A_r, \mu=\tilde \mu} -\Sigma(\mu)]+\Sigma(\delta_{S^1})
=
\inf_{\mu\in A_r, \mu=\tilde \mu} -\Sigma(\mu)
\,.$$

For $\mu\in A_r$ with $\mu=\tilde \mu$, write
$\mu=\rho(dr)\times d\theta$, with $\rho\in \mcal([0,r])$. Then,
\begin{eqnarray*}
\Sigma(\mu)&=&\frac{1}{4\pi^2}\iint_{0}^{2\pi}
\iint_{0}^r \log|se^{i\theta}-s'e^{i\theta'}|\rho(ds)\rho(ds')
\\
&=&
\frac{1}{2\pi}\int_{0}^{2\pi}
\iint_{0}^r \log|s-s'e^{i\theta'}|\rho(ds)\rho(ds')\\
&=& -\iint_0^r (\log s')[2{\bf 1}_{s'>s}+{\bf 1}_{s=s'}]\rho(ds)\rho(ds')
\,,
\end{eqnarray*}
where we used \eqref{patsats} in the last equality. The last
expression is maximized (over $\rho\in\mcal([0,r])$) at
$\rho_r=\delta_r$.

\end{proof}

\subsection{Discussion of the proof}

Functions
somewhat similar to (\ref{IGREEN}) or (\ref{IHKLOCAL})  arise as
rate functions in large deviations problems for empirical
measures of eigenvalues of random matrices (see e.g.
\cite{BG,BZ}). In particular, much of the analysis of the energy
term $\ecal_h(\mu)$ can be carried over from the eigenvalue
setting and from known results in classical weighted potential
theory  on $\C$ (see \cite{ST}). We recall that  the weighted
equilibrium measure of $K$ with respect to the weight $e^{- Q}$ is
the unique maximizer in $\mcal(K)$ of the weighted energy
function,
\begin{equation} \label{SIGMAQK} \Sigma_{Q, K}(\mu) = \int_K
\int_K \log \left(|z - w| e^{- Q(z)} e^{- Q(w)}\right) d\mu(w)
d\mu(z). \end{equation} We observe that $Q = \frac{\phi}{2}$ in
the global setting, i.e. the weight is essentially a Hermitian
metric.

However, the (non-differentiable) sup function $J^{h, K}(\mu) :
= \sup_K U^{\mu}_{h}$ is quite different from, and somewhat more
difficult than,
 the linear functions such as $\int x^2 d\mu$ which occur in
the eigenvalue setting.
Under Assumption (\ref{REGULAR}), we
show that it is convex and continuous on $\mcal(\CP^1)$ with
respect to weak convergence (which, due to the compactness of
$\CP^1$, is equivalent to  vague convergence) of
probability measures (Proposition \ref{LSCJG}).
The continuity also uses the fact that
the Green's function $G_h$ is bounded above on  $\CP^1$.

It is not obvious that the minimizer of $I^{h, K}$ should be the
same as the minimizer of (\ref{SIGMAQK}). This is proved in
Proposition \ref{GLOBALMAX}. The main differences are that (i)
$\ecal_h$ is not constrained to measures supported on $K$; (ii)
the second $\sup_K$ term is additional to the energy. In
Proposition \ref{GLOBALMAX}, we show that $\nu_{h, K}$ minimizes
both $- \ecal_h$ and $J^{h, K}$.

Besides potential theory, an important ingredient in the proof of
Theorem \ref{g0} is a formula for the joint probability current of
zeros when $\pcal_N$ is endowed with the Gaussian measure derived
from the inner product (\ref{INNERPRODUCT}). A novelty of our
presentation is that we derive the joint probability currently in
a natural way from the associated Fubini-Study probability measure
on $\PP \pcal_N$.
In the following,  we work in the standard chart, i.e.
$(\C)^{(N)}$. Let
 $\Delta(\zeta) = \prod_{i < j} (\zeta_i - \zeta_j)$
 denote the Vandermonde determinant,  $s_\zeta(\cdot)\in \pcal_N$
 the polynomial with zero set $\{\zeta_1,\ldots,\zeta_N\}$,
 and $d^2 \zeta = d\zeta \wedge d \bar{\zeta}$ on $\C$.

 In the following, and hereafter, we often use the following
 identity:

 \begin{equation} \label{LOGMUZ}
 \int_{\CP^1} G_{h}(z,w) dd^c \log ||s_{\zeta} (w)||_{h^N}^2 = N
\int G_h (z,w) d\mu_{\zeta}(w) = N U^{\mu_{\zeta}}_h(z),
\end{equation} which follows from the definitions (\ref{ZN}) and (\ref{GHDEF})
and from the Poincar\'e-Lelong formula (see (\ref{PL})).

\begin{prop} \label{FSVOLZETA2intro}
The joint probability current (\ref{JPCDEF}) (see also
(\ref{JPD})) is given in the affine chart $(\C)^N \subset
(\CP^1)^N$ by
\begin{eqnarray}
    \label{eq-030209b}
    \vec K_n^N(\zeta_1, \dots, \zeta_N) & = & \frac{1}{Z_N(h)}
\frac{|\Delta(\zeta_1, \dots, \zeta_N)|^2  d^2 \zeta_1 \cdots d^2
\zeta_N}{\left(\int_{\CP^1} \prod_{j = 1}^N |(z - \zeta_j)|^2 e^{-
N \phi(z)}  d\nu(z) \right)^{N+1}}
 \\
\label{eq-030209d} & = & \frac{1}{\hat{Z}_N(h)} \frac{\exp \left(
\sum_{i < j} G_{h}(\zeta_i, \zeta_j) \right) \prod_{j = 1}^N e^{-
2 N \phi(\zeta_j)} d^2 \zeta_j  }{\left(\int_{\CP^1} e^{N
\int_{\CP^1} G_{h}(z,w) d\mu_{\zeta}} d\nu(z) \right)^{N+1}}.
\end{eqnarray}
where  $G_h$  is the Green's function (\ref{GHDEF}). Also,
$Z_N(h)$ and  $\hat{Z}_N(h)$  are  normalizing constants,
defined so that the measure on the left side  has
mass one.
\end{prop}
The second expression (\ref{eq-030209d}) is invariantly defined.
We will see that
\begin{equation} \label{NORMALIZINGCONSTANTS} \left\{ \begin{array}{ll} Z_N(h) = |\det \acal_N(h)|^{-2} =
Vol(\bcal^2(L^N, h^N)), & \mbox{see}\;  (\ref{ACALH})
\\ & \\
\hat{Z}_N(h)  = |\det \acal_N(h)|^{-2} e^{- (-\frac{1}{2}N (N - 1)
+N (N + 1)) E(h)}, & \mbox{see }  \;  (\ref{EH}). \end{array}
\right.
\end{equation}
Here, $\acal_N(h)$ is the change of basis matrix from the
monomials $z^j$ to an orthonormal basis for the inner product $G_N
(h, \nu)$ on $H^0(\CP^1, \ocal(N))$. In Lemma \ref{APPROXRATE}, we
further rewrite the expression for $\vec K_n^N(\zeta_1, \dots,
\zeta_N)$ as a functional $I_N^{h, \nu} (\mu_{\zeta})$ on the
measures $\mu_{\zeta}$. The rate function $I$ is then extracted
from $I_N$ as $N \to \infty$.

To complete the calculation, we need to determine the logarithmic
asymptotics of $\hat{Z}_N(h)$.

\begin{lem} \label{NC}
 We have, $$\lim_{N \to \infty}
\frac{1}{N^2} \log \hat{Z}_N(h) = \frac12
\log \mbox{Cap}_{\omega_h} (K).
$$
    \end{lem}

  This limit formula gives an alternative approach to the asymptotics of  $|\det
\acal_N(h)|^2$ from that in  \cite{BB}, in this one dimensional
setting.

\subsection{\label{sec-lowbrow} Sketch of proof for Kac-Hammersley}

  We now
sketch  the proof of Theorem \ref{g0} in the case of the
Kac-Hammersley example. In this case, we do not need   the
geometric language used in the rest of the paper.

Consider the polynomial
$P_N(z)=\sum_{i=0}^N a_i z^i$, where the $a_i$ are independent
Gaussian circular standard
complex random variables. We have
$P_N(z)=a_N \prod_{i=1}^N (z-z_i)$.
Further, conditioned on $a_N$, the variables
$\{b_i\}_{i=0}^{N-1}$ are
independent, circular normal, of variance $|a_N|^{-2}$.

Let $\Delta=\prod_{i< j} |z_i-z_j|$. Then, the Jacobian of the
transformation $\{b_i\}_{i=0}^{N-1}\mapsto \{z_i\}_{i=1}^N$ is
$|\Delta|^2$. On the other hand, with $d\mu_\xi\in\mcal(\C)$
denoting the empirical measure of the zeros of $P_N$,
\begin{eqnarray}
\label{eq-1}
&&|a_N|^2+\sum_{i=0}^{N-1}
 |a_N b_i|^2=\sum |a_i|^2  \\
&=&
(2\pi)^{-1}\int_{S^1} P_N(z)P_N^*(z) dz
\nonumber \\
&=&
\frac{|a_N|^2}{2\pi}
\int_{S^1} \prod|z-z_i|^2 dz
=
\frac{|a_N|^2}{2\pi}
\int_{S^1} e^{2 N \int  \log|z-x|d\mu_\xi(x) } dz
:= |a_N|^2 e^{N{\jcal}_N(\mu_\xi)}\,,\nonumber
\end{eqnarray}
where the integrals are path integral along the unit circle,
and we used the fact that the integrand is real to express
it as an exponential of a real function.
We then have, for any measurable set $A\subset \mcal(\C)$,
\begin{eqnarray}
    \label{eq-200206a}
    \mbox{\bf Prob}_N(d\mu_\xi\in A) &=&\frac{1}{Z_N}
    \int_{z_i: \Delta\neq 0} dz_1\ldots dz_N {\bf 1}_{\{L_N\in A\}}
    \int_0^\infty  y^{N}
    e^{N^2 \Sigma(L_N)- y e^{N {\jcal}_N(L_N)}} dy\nonumber \\
&=&
\frac{1}{\widetilde{Z_N}}
    \int_{z_i: \Delta\neq 0} dz_1\ldots dz_N {\bf 1}_{\{L_N\in A\}}
    e^{N^2  \Sigma(L_N)} e^{-N^2 {\jcal}_N(L_N)}.
\end{eqnarray}
Here, $L_N=N^{-1}\sum_{i=1}^N \delta_{z_i}$,
$\Sigma$ is as in \eqref{SIGMADEF} (with the convention
$\log(0)=0$),
and $Z_N$, $\widetilde{Z_N}$ are normalization constants.

Now, for each fixed $\mu$ for which $\log|z-\cdot|$ is
uniformly integrable for $z\in S^1$, we have that
\begin{equation}
    \label{eq-33}
{\jcal}_N(\mu)=N^{-1}\log
\left(\frac{1}{2\pi}\int_{S_1} \exp(2N \langle \mu,
\log |z-\cdot|\rangle) dz\right)\to_{N\to\infty}
2J(\mu)\,,
\end{equation}
where
$$J(\mu):=  \max_{z\in S^1} \int \log |z-x|d\mu(x)\,.$$
One thus expects, as in \cite{BG,BZ},
to obtain the large deviation principle,
with speed $N^2$ and rate function
    $$2J(\mu)-\Sigma(\mu)-\inf_{\nu \in \mcal(\C)}
    [2J(\nu)-\Sigma(\nu)]\,.$$
    (Compare with
\eqref{IHKLOCAL}, noting that $K=S^1$ and $\phi=1$ on $S^1$ for the
Kac-Hammersley example.)
The technical details of the derivation, however, are best handled
in a more general geometric framework, where relevant properties
of the rate function are more transparent.

\subsection{Generalizations}
We close the introduction with some comments on the generalization
of the results of this article to other \kahler manifolds.  In the
sequel
 \cite{Z}, we use the method of this article to give an explicit
 formula for the joint probability current in the more difficult
 higher genus cases.  In higher genus, the relation between
configuration spaces and sections of line bundles of degree $N$ is
the subject of  Abel-Jacobi theory,  and the formula for the joint
probability current involves such objects as the prime form.  Some
of the geometric discussion of this paper is intended to set the
stage for the higher genus sequel. The results can also be
generalized from  Gaussian ensembles to non-linear ensembles of
Ginzburg-Landau type. For the sake of brevity, we do not carry out
the generalization here.

Another type of generalization to consider is to ensembles of
random holomorphic functions, for instance random holomorphic
functions in the unit disc with various weighted norms or random
entire functions on $\C$. The random analytic functions have an
infinite number of zeros and one apparently needs to make a finite
dimensional approximation to obtain a useful configuration space
and a map to empirical measures.

An interesting question is whether one can  generalize the large
deviations results to higher dimensions. One could consider the
hypersurface zero set of a single random section, or the joint
zero set of a full system of $m$ sections in dimension $m$. The
rate function
$\tilde I^{h, K}$ has a generalization to all dimensions,
and so the large deviations result might admit a generalization.
But the approach of this article, to extract the large deviations
rate function from the joint probability current of zeros, does not
seem to generalize well to higher dimensions. In dimension one,
there exists a simple  configuration space of possible zero sets
of sections,   but in higher dimensions there is no manageable
analogue. It is possible that one can avoid this impasse by
working on the space of potentials $\frac{1}{N} \log ||s||_{h^N}$
rather than on the configuration space of zeros.  But it appears
that one would have to extract the rate function directly from
the potentials without using zeros coordinates. This circle of
problems fits in very naturally with the \kahler potential theory
of \cite{GZ,BB,Ber1,Ber2}, and it would be interesting to explore
it further.

Finally, the authors would like to thank R. Berman,  T. Bloom, A.
Dembo, A. Guionnet and B. Shiffman for helpful conversations.

\section{\label{BACKGROUND} Background}

Polynomials of degree $N$ on $\C$ may be viewed as meromorphic
functions on $\CP^1$ with a pole of order $N$ at $\infty$, or equivalently
as holomorphic sections of the $N$ power $\ocal(N) \to
\CP^1$ of the hyperplane line bundle $\ocal(1) \to \CP^1$, or again as homogeneous holomorphic polynomials
of degree $N$ on $\C^2$. It is useful to employ the geometric language of line
bundles, Hermitian metrics and curvature, and in \cite{Z} this
language is indispensible. We briefly recall the relevant
definitions, referring the reader to \cite{GH} for further details.

We use the following standard notation: $\frac{\partial}{\partial
z} = \frac{1}{2} (\frac{\partial}{\partial x} - i
\frac{\partial}{\partial y} ),  \frac{\partial}{\partial \bar{z}}
= \frac{1}{2} (\frac{\partial}{\partial x} + i
\frac{\partial}{\partial y} ). $ Also, $\partial f =
\frac{\partial f}{\partial z} d z$ and similarly for $\dbar f$.
The Euclidean Laplacian is given by $\Delta = 4 \;
\frac{\partial^2}{\partial z
\partial \bar{z}}$ and $\ddbar = \frac{\partial^2}{\partial z
\partial \bar{z}} d z \wedge d \bar{z}.$
It is often convenient to use  the real operators  $d =
\partial + \dbar, d^c := \frac{i}{4 \pi} (\dbar -
 \partial)$ and   $dd^c = \frac{i}{2\pi}
 \ddbar$. Thus, $dd^c f = \frac{ i}{8 \pi} \Delta f dz \wedge d\bar{z} = \frac{1}{4 \pi} \Delta f dx \wedge dy. $
We will often need the classical formula, \begin{equation}
\label{FUNDSOL} \Delta \left(\frac{1}{2 \pi} \log |z|\right)  =
\delta_0 \iff dd^c (2 \log |z|) = \delta_0 dx \wedge dy.
\end{equation}
Henceforth, we regard $\delta_0$ as a $(1,1)$ current so that
$\delta_0$ and $\delta_0 dx \wedge dy$ have the same meaning.

A smooth Hermitian metric $h$ on a holomorphic line bundle $L$ is
a smooth family $h_z$ of Hermitian inner products on the
one-dimensional complex vector spaces $L_z$. Its Chern form is
defined by
\begin{equation}\label{curvature} c_1(h)= \omega_h : = -\frac{\sqrt{-1}}{2 \pi}\ddbar \log \|e_L\|_h^2\;,\end{equation} where $e_L$ denotes a local
holomorphic frame (= nonvanishing section) of $L$ over an open set
$U\subset M$, and $\|e_L\|_h=h(e_L,e_L)^{1/2}$ denotes the
$h$-norm of $e_L$. We say that $h$ is positive if the (real)
2-form $\omega_h $  is a positive $(1,1)$ form,  i.e.  defines a
\kahler metric. For any smooth Hermitian metric $h$ and local
frame $e_L$ for $L$, we  write $\|e_L\|_h^2 = e^{-\phi} $ (or,  $h
= e^{- \phi}$), and
\begin{equation} \label{DDCPHI} \omega_h = \frac{\sqrt{-1}}{2 \pi} \ddbar \phi  = dd^c
\phi. \end{equation} We  refer to $\phi = - \log ||e_L||_h^2$ as
the potential of $\omega_h$ in $U$, or as the \kahler potential
when $\omega_h$ is a \kahler form.
 We are interested in
general smooth metrics, not only those where $\omega_h$ is
positive; for instance, our methods and results apply in the case
where
 $\phi = 0$ (i.e. the metric is flat) on the support of  $d\nu$.
The metric $h$ induces Hermitian metrics $h^N$ on
$L^N=L\otimes\cdots\otimes L$ given by $\|s^{\otimes
N}\|_{h^N}=\|s\|_h^N$.
 The
 $N$-dependent factor $e^{- N \phi }$ is then the local expression of   $h^N$
 in the local frame $e^N$. We will only be considering the line bundles $\ocal(N) \to \CP^1$
 in this paper.

We now specialize to the hyperplane line bundle $\ocal(1) \to
\CP^1$ and its powers. We recall that $\CP^1$ is the set of lines
through $0$ in $\C^2$. The line through $(z_0, z_1)$ is denoted
$[z_0, z_1]$, which are the the homogeneous coordinates of the
line.
 In the case of $\CP^1$ there exists a single holomorphic line
 bundle $L^N$ of each degree. One writes $L = \ocal(1)$ and $L^N =
 \ocal(N)$. The bundle $\ocal(1)$ is dual to the tautological line
 bundle $\ocal(-1) \to \CP^1$ whose fiber at a point $[z_0, z_1] \in \CP^1$
 is the line $[z_0, z_1]$ in $\C^2$.  The line bundle $\ocal(1)$
  is defined by two charts $U_1 = \CP^1 \backslash \{\infty\}$
 $(z_0 \not= 0)$
 and $U_2 = \CP^1 \backslash \{0\}$ ($z_1 \not= 0$). A frame (nowhere vanishing
 holomorphic section) of $\ocal(-1)$ over $U_1$ is
given  by $e_1^*([z_0, z_1])= (1, \frac{z_1}{z_0})$, and over
$U_2$ by $e_2^*([z_0, z_1]) = (\frac{z_0}{z_1}, 1)$. The dual
frames are the homogeneous polynomials on $\C^2$ defined by
$e_1(z_0, z_1) = z_0, $ resp. $e_2(z_0, z_1) = z_1.$

 The  potential $\phi$  is only defined relative to a frame, and we will need
 to know how it changes under a change of frame.
  Suppose that $\phi_1$ is the potential
 of $\omega_h$ in the frame $e_1$, i.e. $||e_1([z_0, z_1])||_h^2 = e^{- \phi_1}$. We assume that $h$, hence
 $\phi_1$ is smooth in $U_1$ and we may (with a slight abuse of notation) regard it as
 a function on $U_1$ or on $\C$ in the standard coordinate $[z_0, z_1] \to \frac{z_0}{z_1} = w$.
 In the frame $e_2$ we have the local  potential  $||e_2(([z_0, z_1]) )||_h^2  = e^{- \phi_2}$ for
some $\phi_2 \in C^{\infty}(\CP^1 \backslash \{0\})$, which we
identify with a function on $\C$. On the overlap $\CP^1 \backslash
\{0, \infty\}$ the frames $e_1, e_2$ are related by $e_2 ([z_0,
z_1]) = \frac{z_1}{z_0} e_1 ([z_0, z_1]) $, so $||e_2([z_0, z_1])
||_h^2 = |\frac{z_1}{z_0}|^2 ||e_1([z_0, z_1]) ||_h^2 $. It
follows that $\phi_2 ([z_0, z_1])  = \phi_1([z_0, z_1]) - 2 \log
|\frac{z_1}{z_0}|.$ If we use $w = \frac{z_0}{z_1}$ as a local
coordinate, then $\phi_2(w) = \phi_1(\frac{1}{w}) +  \log |w|^2$.
As an illustration, the \kahler potential of the Fubini-Study
metric on $\ocal(1)$ is given by $\log (1 + |w|^2) = \log (1 +
\frac{1}{|w|^2}) + \log |w|^2$ in the two charts.

An important observation in understanding the global nature of
(\ref{eq-030209d}) is the following:

\begin{lem} \label{INVAR} The $(1,1)$ form $e^{- 2 \phi_1(z)} dz
\wedge d\bar{z}$ in the chart $U_1$ extends to a global smooth
$(1,1)$ form  $\kappa$ on $\CP^1$. In the chart $U_2$ it equals $e^{- 2
\phi_2(z)} dz \wedge d\bar{z}$. \end{lem}

\begin{proof}

We need to check its invariance under the change of variables
$\sigma(z) = \frac{1}{z}$. We have,
$$\sigma^* e^{- 2 \phi_1(z)} dz
\wedge d\bar{z} = e^{- 2 \phi_1(\frac{1}{z})} \frac{d z \wedge
d\bar{z}}{|z|^4}. $$ Since $ \phi_1(\frac{1}{z}) = \phi_2(z) -
\log |z|^2,$ this is
$$e^{- 2 \phi_2(z)} e^{2 \log |z|^2} \frac{d z \wedge
d\bar{z}}{|z|^4} = e^{- 2 \phi_2(z)}  d z \wedge d\bar{z} . $$

\end{proof}

\subsection{\label{PLLDEF} Poincar\'e-Lelong formula for the empirical measure of
zeros}

The empirical measure of zeros $Z_s$  (\ref{ZN}) is given by
  (one-dimensional) Poincar\'e-Lelong formula,
\begin{equation} \begin{array}{lll}
\label{PL} Z_s   =  \frac{i}{\pi N} \ddbar \log |f| & =
&\frac{i}{N
\pi} \ddbar \log \|s\|_{h^N} + \omega_h \\ && \\
& = & \frac{2}{N} dd^c \log \|s\|_{h^N} + \omega_h  \;.
\end{array}\end{equation} It is completely elementary in dimension
one.

 \subsection{$dd^c$ Lemma}

 We will need the $dd^c$ Lemma on not-necessarily-positive $(1,1)$
 currents. The $dd^c$ Lemma on forms (cf. \cite{Dem}, Lemma 8.6 of
 Chapter VI) asserts that on a compact \kahler manifold, a
 $d$-closed $(p,q)$ form $u$ may be expressed as $u = dd^c v$
 where $v$ is a $(p-1, q-1)$ form. The same Lemma is true for
 currents, with the change that $v$ is only asserted to be a
 current.

 When $\omega, \omega'$ are two cohomologous positive closed
 $(1,1)$ currents (which on $\CP^1$ simply means
 $\int_{\CP^1} \omega = \int_{\CP^1} \omega'$),
  then one has a regularity theorem: $\omega -
 \omega' = dd^c \psi$ where $\psi \in L^1(\CP^1, \R)$. We refer to
 \cite{GZ}, Proposition 1.4.

\subsection{\label{GFS} Hermitian inner products and Gaussian measures  on $H^0(\CP^1, \ocal(N))$}

We denote by $H^0(\CP^1, \ocal(N))$ the space of holomorphic
sections of $\ocal(N)$. It is well-known that they correspond to
polynomials of degree $N$, which are their local expressions in
the affine chart $U = \CP^1 \backslash \{\infty\}$ (see
\cite{GH}).

As mentioned in the introduction, the data $(h, \nu)$ determine
inner products $G_N(h, \nu)$ on the complex vector spaces
$H^0(\CP^1, \ocal(N))$ (see (\ref{INNERPRODUCT}) and
(\ref{INNERPRODUCTa})). An  inner product on $H^0(\CP^1,
\ocal(N))$ induces  a Gaussian measure on this complex vector
space by the formula,
\begin{equation}\label{gaussian}d\gamma_N(s_N):=\frac{1}{\pi^m}e^
{-|c|^2}dc\,,\quad s_N=\sum_{j=1}^{d_N}c_jS^N_j\,,\quad
c=(c_1,\dots,c_{d_N})\in\C^{d_N}\,,\end{equation} where $d_N=N+1$,
$\{S_1^N,\dots,S_{d_N}^N\}$ is an orthonormal basis for
$H^0(\CP^1, \ocal(N))$, and $dc$ denotes $2d_N$-dimensional
Lebesgue measure.   The measure $\ga_N$ is characterized by the
property that the $2d_N$ real variables $\Re c_j, \Im c_j$
($j=1,\dots,d_N$) are independent Gaussian random variables with
mean 0 and variance $1/2$; equivalently,
$$\E_N c_j = 0,\quad \E_N c_j c_k = 0,\quad  \E_N c_j \bar c_k =
\de_{jk}\,,$$ where $\E_N$ denotes the expectation with respect to
the measure $ \ga_N$.

In \S \ref{FSDEF}, we will define an essentially equivalent
Fubini-Study volume form on the projective space of sections $\PP
H^0(\CP^1, \ocal(N))$.

\section{\label{FSDEF} Joint probability current of zeros and the Fubini-Study volume form}

In this section, we define the principal object of this article,
the joint probability current of zeros. We then prove the first
part \eqref{eq-030209b} of Proposition \ref{FSVOLZETA2intro},
giving the
 formula for  the joint probability current of zeros as the
pull back to configuration space of the Fubini-Study volume form
on the projective space of sections. We the re-write the formula
in terms of the Green's function   in the next section.

\subsection{\label{PLL} The joint probability current of zeros}

The joint probability current of zeros is defined by
\begin{equation}\begin{array}{lll}  \vec K_{N}^N(z^1, \dots, z^{N}):& = & \E(Z_s(z^1) \otimes
Z_s(z^2) \otimes \dots \otimes Z_s(z^{N})).  \end{array}
\end{equation}

It is a current on the configuration space $(\CP^1)^{(N)}$ of $N$ points.
It is the extreme case
 $n = N$ of the $n$-point zero correlation current
\begin{equation} \label{JPCURRENT} \vec K_n^N(z^1, \dots, z^n):= \E(Z_s(z^1) \otimes
Z_s(z^2) \otimes \dots \otimes Z_s(z^n)) \end{equation} on the
configuration space $(\CP^1)^{(n)}$. Recall that
by a current we mean a linear
functional on test forms, i.e. for any test function $\phi_1(z^1)
\otimes \dots \otimes \phi_n(z^n) \in C ((\CP^1)^{(n)}$,
\begin{equation} \big(\vec K_n^N(z^1, \dots, z^n),
\phi_1(z^1) \otimes \dots \otimes \phi_n(z^n)\big) = \E
\left[\big( Z_s,\phi_1\big)\big( Z_s,\phi_2\big)\cdots \big(Z_s,
\phi_n\big)\right].
\end{equation}

\subsection{Fubini-Study formula}

We now present the most useful approach to the joint probability
current of zeros in the case of genus zero.

It is a classical fact that the projective space of sections $\PP
H^0(\CP^1, \ocal(N))$ may be identified with the configuration
space $(\CP^1)^{(N)}$ of $N$ points of $\CP^1$. This essentially
comes down to the elementary fact that a set $\{\zeta_1, \dots,
\zeta_N\}$ determines a line of polynomials $[P_{\zeta}] \in \PP
\pcal_N$ of degree $N$, at least when none of the  zeros occur at
$\infty$.   Viewed as holomorphic sections of $\ocal(N) \to \CP^1$
one can also allow $\infty$ to be a zero and then $N$ points of
$\CP^1$ corresponds to a line of holomorphic sections.

The correspondence $\zeta \to [P_{\zeta}]$ defines a line bundle
\begin{equation} \label{ZCALN} \zcal_N \to (\CP^1)^{(N)}, \;\; (\zcal_N)_{\zeta} = \{[p] \in \pcal_N: \dcal(p) = \zeta \}, \end{equation}
i.e.  the fiber of $\zcal_N$ at $\zeta_1 + \cdots + \zeta_N$ is
the line $\C P_{\zeta}$ of holomorphic sections of $\ocal(N)$ with
the divisor $\zeta = \zeta_1 + \cdots + \zeta_N$. It  is
isomorphic to the bundle $\ocal(1) \to \PP H^0(\CP^1, \ocal(N))$
under  the identification $\PP H^0(\CP^1, \ocal(N))
=(\CP^1)^{(N)}$. One can construct a form representing the first
Chern class $c_1(\zcal_N)$ using a Hermitian inner product on
$\zcal_N$ or equivalently a Hermitian inner product on $H^0(\CP^1,
\ocal(N))$: at a point $\zeta \in (\CP^1)^{(N)}$, the $G_z$-norm
of a vector $P_{\zeta} \in \zcal_{\zeta}$ is $||P_{\zeta}||_G$,
the norm of $P_{\zeta}$ as an element of $H^0(\CP^1, \ocal(N))$.
This is  the Fubini-Study Hermitian metric determined by the inner
product.

 Let us recall the basic
definitions and formulae in the case of the standard inner product
on $\C^{d + 1}$ and $\CP^d$. Let $Z \in \C^{d + 1}$ and let
$||Z||^2 = \sum_{j = 1}^{d + 1} |Z_j|^2. $ In the open dense chart
$Z_0 \not= 0$, and
 in affine coordinates $w_j =
\frac{Z_j}{Z_0}$, the Fubini-Study  volume form is given  by,
$$dVol_{I} = \frac{\prod_i dw_i \wedge d\bar{w}_i}{(1 + ||w||^2)^{d + 1}}. $$
For our purposes, it is more useful to lift this form to
$\C^{d+1}$ under the natural projection, $\pi: \C^{d + 1} - \{0\}
\to \CP^d$.  A straightforward calculation shows that
\begin{equation} \pi^* dVol_{I} = ||Z_0||^2 \frac{
\prod_{j = 1}^{d } dZ_j \wedge d \bar{Z}_j}{||Z||^{2 (d + 1)}},
\end{equation} in the sense that
$$\frac{dZ_0 \wedge d\bar{Z}_0}{||Z_0||^2} \wedge \frac{\prod_{i = 1}^d  dw_i \wedge d\bar{w}_i}{(1 + ||w||^2)^{d +
1}} = \frac{ \prod_{j = 0}^{d } dZ_j \wedge d \bar{Z}_j}{||Z||^{2
(d + 1)}}. $$

We need a more general formula where the inner product $||Z||^2$
is replaced by any Hermitian inner product on $\C^{d + 1}$. We
recall that the space of Hermitian inner products on $V$ is the
symmetric space $GL(d + 1, \C)/ U(d + 1)$. If we identify $V =
\C^d$ and fix the standard inner product $(v, w)$, then any other
inner product has the form $G(v, w) = (P v, w)$ where $P$ is a
positive Hermitian matrix. It has the form $P = A^*A$ where $A \in
GL(d + 1, \C)$.

 Suppose, then, that instead of the standard inner norm $||Z||$ on $\C^{d + 1}$
we are given the norm $||A Z||$ where $A \in GL(d + 1, \C)$.  Then
the Fubini-Study metric becomes  $\ddbar \log ||A Z||^2$. Since the
linear transformation defined by  $A$ is holomorphic, the
associated volume form $dV_A$ is simply the pull-back by $A$ of
the previous form,
\begin{eqnarray}\label{VOLA}
 \pi^* dVol_{A} &=& A^* \frac{(\ddbar \log ||Z||^2)^{d
+ 1}}{\frac{dZ_0 \wedge d\bar{Z}_0}{||Z_0||^2}}  =  \frac{||(A
Z)_0||^2 }{||A Z||^{2 (d + 1)}} A^* \left( \frac{ \prod_{j = 0}^{d
} dZ_j \wedge d \bar{Z}_j}{dZ_0 \wedge d\bar{Z}_0}\right) \\
& = &| \det A|^2 |(A Z)_0|^2 \cdot \left( \frac{\partial}{\partial
Z_0} \wedge \frac{\partial}{\partial \bar{Z}_0} \bot A^* (dZ_0
\wedge d\bar{Z}_0)  \right)^{-1}  \left( \frac{ \prod_{j = 1}^{d }
dZ_j \wedge d \bar{Z}_j}{{||A Z||^{2 (d + 1)}}}\right).
\nonumber
 \end{eqnarray}

We now prove the first part of Proposition \ref{FSVOLZETA2intro}.

\subsection{\label{DETG0} Proof of \eqref{eq-030209b} in
Proposition \ref{FSVOLZETA2intro}}

To prove  \eqref{eq-030209b}, we use (\ref{VOLA}) and   change
variables to zeros coordinates.

We first consider the change of variables in local coordinates on $\CP^1$. We fix the usual affine
chart $U \subset \C$ and let $z$ be the local coordinate. We then have a corresponding local coordinate
system $(\zeta_1, \dots, \zeta_N)$ on $(\CP^1)^N$ which is defined in the chart $(\C)^N$.

We have defined the joint probability current (\ref{JPCURRENT}) as an $(N, N)$ form on configuration
space $(\CP^1)^N$. It pulls back under the $S_N$ cover $(\CP^1)^N \to (\CP^1)^{(N)}$ and we wish to express it
in the local coordinate system $(\zeta_1, \dots, \zeta_N)$ to obtain the formula
in Proposition \ref{FSVOLZETA2intro}. We then write down its density
with respect to the local Lebesgue volume form $d^2 \zeta_1 \cdots d^2 \zeta_N$ of the chart.

To prove the Proposition, we start with the    Newton-Vieta's formula:
\begin{equation} \label{VIETA}  \prod_{j = 1}^N (z - \zeta_j) =
 \sum_{k = 0}^N (-1)^k e_{N - k} (\zeta_1, \dots, \zeta_N)\;\; z^k.
\end{equation}
Here, the  elementary symmetric functions are defined by $$e_j =
\sum_{1 \leq p_1 < \cdots < p_j \leq N} z_{p_1} \cdots z_{p_j}.$$

As mentioned above, the  formula (\ref{VIETA}) defines a map
$(\CP^1)^{(N)} \to \pcal_N$,  which is a section of  the line
bundle $\zcal_N$ (\ref{ZCALN}). It is the section  taking its
values in the polynomials $\sum_{i=0}^N a_i z^i$ for which $a_N =
1$. Since $e_0(\zeta) \equiv 1$,  the linear coordinates are
affine coordinates in the chart $c_0 = 1$, where $c_j$ are
coordinates with respect to the basis $\{z^j\}$.
 We then  change variables from the Lebesgue volume form
$da_1 \wedge d \bar{a}_1 \wedge \cdots \wedge da_{N} \wedge
d\bar{a}_N$ in the affine chart to a volume form in the
coordinates $(\zeta_1, \dots, \zeta_N)$. It is well-known (see
e.g. \cite{LP}) that this change of variables has Jacobian
$|\Delta(\zeta)|^2$ where as above,  $ \Delta(\zeta_1, \dots,
\zeta_N) = \prod_{1 \leq j < k \leq N} (\zeta_k - \zeta_j)$ is the
Vandermonde determinant.

We now express the Fubini-Study  probability measure on $\PP
H^0(\CP^1, \ocal(N))$ in the coordinates $\zeta_j$. The first
problem we face is that the right side of (\ref{VIETA}) expresses
the polynomial on the left side in coordinates with respect to the
basis $\{z^j\}_{j = 0}^N$, which is usually  not an  orthonormal
basis with respect to the inner product (\ref{INNERPRODUCT}). We
need to make the additional change of variables from coordinates
$E_j$  with respect to an orthonormal basis $\{\psi_j\}$ for our
inner product $G_N(h, \nu)$, $ \prod_{j = 1}^N (z - \zeta_j) =
\sum_{\ell = 0}^d  \ecal_{N - \ell} \psi_{\ell}$  to coordinates
$Z_j = (-1)^{N-j} e_{N - j}$ with respect to the monomial basis
$\{z^j\}$.  With no loss of generality, we  assume that the
orthogonal polynomials $\{\psi_j\}$ are enumerated according to
degree, so that $\psi_N$ is the unique polyomial in the basis with
a $z^N$ term. The change of basis matrix $\acal_N(h, \nu) Z =
\ecal$ is given by,
\begin{equation} \label{ACALH} \left(\acal_N^{jk } \right)_{j, k = 0}^N = \left( \langle
z^j, \psi_k \rangle_{G_N(h, \nu)} \right)_{j,k = 0}^N.
\end{equation}

Next we observe  that   \begin{equation} \label{ANZZBAR}
\frac{\partial}{\partial Z_0} \wedge \frac{\partial}{\partial
\bar{Z}_0} \bot \acal_N(h, \nu)^* dZ_0 \wedge d\bar{Z}_0  =
||\acal_N^{00}||^2.  \end{equation} Indeed, $\acal_N^* dZ_0 = \sum_j
\acal_N^{0 j} dZ_j $ and the desired expression is the coefficient
of $d Z_0 \wedge d \bar{Z}_0$ in $ d (\acal_N^*Z)_0 \wedge d
\overline{(\acal_N^*Z)_0}. $ We further observe that
$|\acal_N^{00}|^2$ is a constant independent of $\zeta$. By our
ordering, $\psi_N = k_N z^N + k_{N -1} z^{N - 1} \cdots$ for some
$k_N \not= 0$. Since \begin{equation} \label{kNSTUFF} \prod_j (z -
\zeta_j) = \sum \ecal_{N - j} \psi_j = z^N + e_1(\zeta) z^{N - 1}
+ \cdots, \end{equation} it follows that
\begin{equation} \acal_N^{00} = k_N^{-1}, \;\; \mbox{and that}\;\; 
||(\acal_N(h, \nu) Z)_0||^2 = k_N^{-2}. \end{equation}
Hence, (\ref{ANZZBAR}) equals $ k_N^{-2}, $ and the factors
$|(\acal_N(h, \nu) Z)_0|^2\cdot k_N^2$ cancel.

Combining this evaluation with  (\ref{VOLA}), we see that the pull
back of the Fubini-Study volume form with respect to $G_N(h, \nu)$
to $\C^{N + 1}$ is given by
\begin{equation} \label{COEFF} | \det \acal_N(h, \nu)  |^2
 \left( \frac{ \prod_{j =
1}^{N } dZ_j \wedge d \bar{Z}_j}{{||\acal_N(h, \nu)  Z||^{2 (N +
1)}}}\right).
\end{equation}

We now change variables to zeros coordinates. As mentioned above,
$ \prod_{j = 1}^{d } dZ_j \wedge d \bar{Z}_j = |\Delta(\zeta)|^2
\prod_j d^2 \zeta_j$.  The denominator in (\ref{COEFF}) equals the
sum of the squares of the components of $\acal_N(h, \nu) Z$, which
is $L^2$ norm-squared of $ \prod_{j = 1}^N (z - \zeta_j)$ with
respect to $G_N(h, \nu)$, i.e.
$$||\acal_N(h, \nu)  Z||^{2 (N +
1)} = \left(\int_{\CP^1} \prod_{j = 1}^N  |(z - \zeta_j)|^2 e^{- N
\phi}  d\nu(z) \right)^{N+1}. $$

Further,
$$\begin{array}{lll}| (\acal_N(h, \nu) Z)_0|^2 = |\ecal_0(\zeta)|^2 & = & |\langle \prod_{j =
1}^N (z - \zeta_j), \psi_N \rangle|^2 \\ && \\
& = & \left| \int_{\C}   \prod_{j = 1}^N (z - \zeta_j))
\overline{\psi_N(z)} e^{- N
\phi(z)} d\nu(z) \right|^2 \\ && \\
& = & e^{2 N \int \phi d\mu_{\zeta}} \int_{\C} \psi_N(z) e^{N
\int_M G_h(z,w) d\mu_{\zeta} } e^{- N \phi(z)} d\nu(z)
\end{array}$$

This completes the proof of \eqref{eq-030209b}.

 \qed

We refer to the coefficient of $d^2 \zeta_1 \cdots d^2 \zeta_N$ in
\eqref{eq-030209b} as the joint probability density (JPD) of
zeros:
\begin{equation}\label{JPD}  D_N(\zeta_1, \dots, \zeta_N)= |
    \det  \acal_N(h, \nu) |^2 \frac{|\Delta(\zeta_1, \dots,
\zeta_N)|^2 }{\left(\int_D \prod_{j = 1}^N  |(z - \zeta_j)|^2 e^{-
N \phi}  d\nu(z) \right)^{N+1}}. \end{equation}

\begin{rem} \label{LOCAL}
The elementary symmetric functions $e_j(\zeta)$  of $\zeta =
(\zeta_1, \dots, \zeta_N)$ are natural coordinates in  $\C^{(N)}$,
and a natural holomorphic volume form is given by
\begin{equation} \Omega_{C^{(N)}} = d e_1 \wedge \cdots \wedge de_N\end{equation}
while the corresponding $(N, N)$ form is
\begin{equation} \Omega_{C^{(N)}} \wedge \overline{\Omega}_{C^{(N)}} = d e_1 \wedge d \bar{e_1}
 \wedge \cdots \wedge de_N \wedge d \bar{e_N} = |\Delta(\zeta_1, \dots,
\zeta_N)|^2 d^2 \zeta_1 \cdots d^2 \zeta_N. \end{equation}

\end{rem}

\subsection{\label{INTRINSICJPC} Intrinsic formula for the joint probability current}

 The Fubini-Study form  has an intrinsic geometric interpretation as the
curvature form (\ref{curvature}) for the Hermitian line bundle
$\zcal_N \to (\CP^1)^{(N)}$ equipped with its metric $G(h^N, \nu)$. This is of
independent
geometric interest
and we pause to consider it.

 A  local frame  for $\zcal_N$
 (henceforth we drop the $N$ for notation simplicity)  is a
non-vanishing holomorphic selection of a polynomial $P_{\zeta}$
from the line $\C P_{\zeta}$ of polynomials (or more generally,
holomorphic sections of $\ocal(N) \to \CP^1$) with divisor
$\zeta$. The standard choice is to trivialize $\zcal$ over
$(\C)^N$ using the section  $P_{\zeta}(z) = \prod_{j = 1}^N (z -
\zeta_j) e^N(z)$ where $e(z)$ is the standard  affine frame of
$\ocal(1) \to \CP^1$ over $\C$. In this article, the inner product
$G = G_N(h, \nu)$ is defined by (\ref{INNERPRODUCTa}). It follows
that the curvature $(1,1)$ form of $\zcal$ is given by
\begin{equation} \label{OMEGAZCAL} \omega_{\zcal} = \frac{i}{2} \ddbar \log ||
P_{\zeta}||_{G(h, \nu)}, \end{equation} where $\ddbar$ is the
operator on $(\CP^1)^{(N)}$. Thus, \begin{equation}
\label{KPCONFIG} \Phi_N(\zeta) : = \log || P_{\zeta}||_{G(h, \nu)}
\end{equation} is the \kahler potential for the \kahler form of configuration
space, and the volume form is given by
\begin{equation} \label{CONFIGJPD} dV_{FS, G_N(h, \nu)} = \left(\frac{i}{2} \ddbar \log ||
P_{\zeta}||_{G_N(h, \nu)}\right)^N, \end{equation} the $(N, N)$
form defined as the top exterior power of (\ref{OMEGAZCAL}). What
\eqref{eq-030209b} asserts is thus equivalent to
\begin{prop} \label{FSVOLZETA2FS} We have,
$$\left(\frac{i}{2} \ddbar \Phi_N \right)^N = |\det  \acal_N(h, \nu)|^2 |\Delta(\zeta)|^2 e^{- (N +
1)\Phi_N(\zeta)} \Pi_{j = 1}^N d^2 \zeta_j. $$
\end{prop}

This Proposition clarifies in what sense the right hand side is a
well-defined volume form on $(\CP^1)^{(N)}$. Namely, it
corresponds to the choice of the \kahler potential $\Phi_N$, i.e.
the expression of the Hermitian metric $G$ on $\zcal$  in the
local frame $P_{\zeta}$.

\section{\label{GREEN} Green's functions and  the joint
probability current: completion of the
proof of Proposition \ref{FSVOLZETA2intro}}

As discussed in the introduction, it is very helpful to express
the joint probability current and rate function in terms of
global objects on $\CP^1$. In the statement of Theorem \ref{g0},
we expressed $I^{h, K}$ in terms of the Green's function $G_h$. In
this section, we give background on the definition and properties
of Green's function that are needed in the proof of Theorem
\ref{g0}.  The main result  is Proposition \ref{JPDINV}, in which
we express the  joint probability current in terms of Green's
functions, and thus complete the proof of Proposition \ref{FSVOLZETA2intro}.

\subsection{\label{HGF} Green's function for $\omega_h$}

The Green's function $G_h(z,w) $ is defined in (\ref{GHDEF}). We now verify that $G_h$ is well-defined,
that it is smooth outside of the diagonal in $\CP^1 \times \CP^1$ and that its only singularity is a logarithmic
singularity on the diagonal. We sometimes write  $g_z(w) = G(z,w)$ to emphasize that the derivatives in
(\ref{GHDEF}) are in the $w$ variable.
When $\omega_h$ is a \kahler metric, $g_z(w)$ is a special case of the notion of Green's
 current for the divisor $\{z\}$.  For background we refer to \cite{He}, although it only discusses the
 case where $\omega_h$ is a \kahler form. We also refer to \cite{ABMNV} for background on global analysis on
 Riemann surfaces.

When we express the Green's function in the charts $U_1 \times
U_1,$ resp. $U_2 \times U_2$ of $\CP^1 \times \CP^1$, we subscript
$G_h$ accordingly. We also drop the subscript $h$ for simplicity
of notation when the metric is understood.

\begin{prop}  There exists a unique function $G_h(z,w) \in L^1(\CP^1 \times \CP^1)$ solving
the system of equations (\ref{GHDEF}). When $z \not= \infty$, in
the local affine chart $\C$ it is given by (\ref{GHFORMintro}).
Under the holomorphic map $z \to \frac{1}{z}$, we have
$$G_1(\frac{1}{z}, \frac{1}{w})  = G_2(z,w).$$
\end{prop}

\begin{proof}

 Given any $z \in \CP^1$, there exists a section $s_z \in H^0(\CP^1, \ocal(1))$ which vanishes
at $z$. There exists a distinguished section (denoted ${\bf 1}_z(w)$ in \cite{ABMNV})
 which has the Taylor expansion $w - z$ in the standard affine frame and which corresponds to the meromorphic
 function $w - z$.   When $z = \infty$,
$s_{\infty}(w) $ corresponds to the meromorphic function $1$. As a
homogeneous polynomial of degree one in each variable on  $\C^2
\times \C^2$ it is given by $w_1 z_0 - z_1 w_0$.  We view the
two-variable section $s_w(z)$ as a section of $\pi_1^* \ocal(1)
\boxtimes \pi_2^* \ocal(1) \to \CP^1 \times \CP^1$ and equip the
line bundle with the product Hermitian metric $h_z \boxtimes h_w$
(here and in what follows, $\boxtimes$ denotes the exterior tensor
product on $\CP^1 \times \CP^1$). We then claim that (with $E(h)$
defined in (\ref{EH})),
\begin{equation} \label{GREENSECTION} G_{h}(z,w) = \log || s_z(w) ||^2_{h_z \boxtimes h_w} - E(h) \end{equation}
satisfies (i)- (iii) of (\ref{GHDEF}) for all $z$. Both (i) and
(ii) are clear from the formula and from (\ref{PL}).

To prove (iii) and the identity claimed in the Proposition, it is
convenient to use the local affine frames $e_j $ of $\ocal(1) \to
\CP^1$ over the affine charts $U_j$ (see \S \ref{BACKGROUND} for
notation).

\begin{lem} \label{GHFORM}
    There exists a constant $E(h)$ so that,
    in the affine chart $U_j$ ($j = 1, 2)$ and
    all
     $z \in \C$,
\begin{equation} \label{GHFORMa}  G_{j}(z,w) = 2 \log
|z - w| - \phi_j(z) - \phi_j(w) +  E(h),
\end{equation}
and $\int_{\C} G_j(z,w) dd^c \phi_j = 0$.
\end{lem}

Indeed,  in $U_1$ we put $z_0 = w_0 = 1$ and $z_1 = z, w_1 = w$,
and then
\begin{equation} \label{ADDa} \log ||s_z(w)||^2_{h_z \boxtimes h_w} = 2 \log |z - w| -  \phi_1(z) -
 \phi_1(w). \end{equation}
 In $U_2$ we put $z_1 = w_1 = 1$ and $z_0 = z, w_0 = w$ and obtain
 the same expression with $\phi_2$ replacing $\phi_1$. On the
 overlap, the stated identity follows from the fact that
 $$2 \log |\frac{1}{z} -  \frac{1}{w}|  - \phi_1(\frac{1}{z}) -
 \phi_1(\frac{1}{w}) = 2 \log |z - w| - \phi_1(\frac{1}{z}) -
 \phi_1(\frac{1}{w}) - 2 \log |z| - 2 \log |w|, $$
 and the fact that $\phi_2(w) = \phi_1(\frac{1}{w}) +  \log |w|^2$
 (see \S \ref{BACKGROUND}).

 To complete the proof, we need to show that
$  \int_{\CP^1}  \log || z - w||^2_{h_z \boxtimes h_w} \omega_h $
is a constant in $z$. In fact we claim that when $z,w \in U_1$,
then
\begin{equation} \label{LOGINTc}  \int_{\CP^1}  \log || z - w||^2_{h_z \boxtimes h_w} \omega_h   =
-  \int \phi \omega_h -4 \pi  \rho_{\phi}(\infty).
\end{equation}

The calculation of this integral can be done by the integration by
parts formulae in \S \ref{IP}. We use  (\ref{ADDa}) to break up
the integrand into three terms. The second integrates to $-
 \phi(z) \int_{\CP^1} \omega_h = - \phi(z),$
while the third integrates to $-  \int \phi \omega_h = - \int \phi
dd^c \phi$. The first (logarithmic) term is of the form
(\ref{LOGROBIN}):
\begin{equation} \label{INTEGRALa} \int_{\C}  2 \log |z - w|
dd^c \phi_1 = \phi_1(z) -  4 \pi \rho_{\phi_1}(\infty).
\end{equation}
In the full sum, the $\phi_1(z)$ terms cancel, leaving the stated
expression. The same integral holds with $\phi_2$ replaced by
$\phi_1$ if $z, w \in U_2$ by the identity in the Proposition.
This proves the integral formula in all cases.

\end{proof}

As an example of the calculation, the Fubini-Study Green's function is given in the chart $U_1 \times U_1$ by
$G_{FS} (z,w) = 2 \log [z,w]^2 - C,$ where $[z,w] = \frac{|z - w|
}{\sqrt{1 + |z|^2} \sqrt{1 + |w|^2}} $. The
constant $C$ is determined by the condition (iii). To study its behavior when $z = \infty$
we change coordinates $\sigma: z \to \frac{1}{z}, w \to \frac{1}{w}$ and study the behavior at $0$. The
distance $[z,w]$ and Green's function  are invariant under the isometry $\sigma$, so we obtain the same expression after the change
of coordinates.  In particular, in these coordinates, $G_{FS}(\infty, u) = 2 \log |u| -  \log (1 + |u|^2) = \phi(\frac{1}{u})$,
where $\phi_{FS}(w) = \log (1 + |w|^2)$.

\begin{rem}  We note that a local \kahler potential $\phi$ (or a
global relative \kahler potential) is only unique up to an
additive constant. One may normalize $\phi$ by the condition
$\int_{\CP^1} \phi \omega_h = 0$. However, in the above formula we
have not done so. We observe that the  Green's function is (as it
must be) invariant under addition of a constant to $\phi$.

\end{rem}

\subsection{\label{GF1} Green's potential of a measure}

We now return to the Green's potential (\ref{Green's potential})
and Green's energy (\ref{GEN}) of the introduction. Given a real
$(1,1)$ form  $\omega$ on $\CP^1$, we define
 \begin{equation} \label{SHH} SH(\CP^1, \omega): = \{u \in L^1(\CP^1, \R
 \cup \{- \infty\}): dd^c u + \omega   \geq 0\}. \end{equation}
For any closed $(1,1)$ form, the  $\ddbar$ Lemma implies that the
map
\begin{equation} \label{PSIMAP} \psi \to \omega_{\psi} := \omega + dd^c \psi \in \mcal(\CP^1) \end{equation} is surjective
and has only constants in its kernel, i.e.
\begin{equation}\label{SHM}  SH(\CP^1, \omega) \simeq \mcal(\CP^1) \oplus \R.
\end{equation}
The Green's potential (\ref{Green's potential}) of a measure
defines a global inverse to (\ref{PSIMAP}) and is uniquely
characterized as the solution of
\begin{equation} \left\{ \begin{array}{l}
dd^d U^{\mu}_{\omega} = \mu - \omega; \\ \\
 \int_{\CP^1} U^{\mu}_{\omega} \omega = 0.
\end{array} \right. \end{equation}
Any smooth integral $(1,1)$ form $\omega \in H^2(\CP^1, \Z)$ is
the curvature $(1,1)$ form of a smooth Hermitian metric $h$ (see
\S \ref{BACKGROUND}), and we subscript the potential by $h$ rather
than $\omega$. Thus,
\begin{equation} \label{DGCAL} dd^c U^{\mu}_{h} (z) = \mu -
\omega_h.
\end{equation}

We illustrate Green's potentials in the important case where
 $\mu = \mu_{\zeta}$. In Lemma \ref{GREENNORM}, we
essentially wrote the $\omega_h$-subharmonic function $\frac{1}{N}
\log ||s_\zeta(z)||_{h^N} $ with  $s_\zeta = \prod_{j = 1}^N (z -
\zeta_j) e^N(z)$ as a Green's potential. To tie the discussions
together, we note that the special case $\omega = \omega_h$  of
Lemma \ref{GREENNORM} can be reformulated in terms of Green's
potentials as follows:
\begin{lem} \label{NORMSUPHI}  We have,

\begin{itemize}

\item $ \frac{1}{N} \log ||s_\zeta(z)||_{h^N} - \frac{1}{N}
\int_{\CP^1} \log ||s_\zeta||_{h^N}^2 \omega_h =
U_{h}^{\mu_{\zeta}}(z). $

\item $   \int \log ||s_{\zeta} (w)||_{h^N}^2  \omega_h =\int_{\C}
\log ||s_{\zeta} (w)||_{h^N}^2  dd^c \phi =  N (\int \phi
d\mu_{\zeta} - E(h)). $

\item
 Hence
$$ ||s_\zeta(z)||_{h^N}^{\frac{1}{N}} e^{ - \frac{1}{N}
\int_{\CP^1} \log ||s_\zeta||_{h^N}^2 \omega_h} =
e^{U_{h}^{\mu_{\zeta}}}. $$

\end{itemize}
\end{lem}

\begin{proof}  Since
$d\mu_{\zeta} = dd^c \frac{1}{N}  \log ||s_\zeta(z)||_{h^N}^2 +
\omega_h, $
$$\begin{array}{lll}    U_{h}^{\mu_{\zeta}}(z) &: =&
    \int_{\CP^1} G_{h}(z,w) d\mu_{\zeta}(w)\\ && \\ & = &
\int_{\CP^1} G_h(z,w)  (\frac{1}{N} dd^c \log ||s_\zeta(w)||_{h^N}^2 +
\omega_h) \\ && \\ & = &  \int_{\CP^1} G_h(z,w)  \frac{1}{N} dd^c
\log ||s_\zeta(w)||_{h^N}^2 \\ && \\ & = & \frac{1}{N} \log
||s_\zeta||_{h^N}^2(z) - \frac{1}{N} \int_{\CP^1} \log ||s_\zeta||_{h^N}^2(z)
\omega_h.
\end{array}$$
\end{proof}

\subsection{Regularity of Green's functions}

For use in the proof of the large deviation principle, we need the
following regularity result on the  Green's function. In what
follows, $D=\{(z,z): z\in \CP^1\}$.
\begin{prop}
    \label{prop-grub}
    $G_h(z,w) \in C^{\infty}(\CP^1 \times \CP^1 \backslash
D)$, and  in any local chart, near the diagonal it possesses the
singularity expansion,
\begin{equation} \label{HADAMARD} G_h(z,w) = 2 \log |z - w| +
\rho(z) + O(|z - w|)
\end{equation}  where $\rho(z)$ is a smooth function on $\CP^1$ known as the
Robin constant. In particular, $G_h(z, \cdot) \in L^1(\CP^1,
\omega_h)$ for any $z$, and there exists a constant $C_G < \infty$
so that $$\sup_{(z, w) \in \CP^1 \times \CP^1} G(z,w) \leq C_G.$$
\end{prop}

\begin{proof}

 When $\omega$ is a \kahler metric, we may form its Laplacian $\Delta_{\omega}$
 and then the Green's function $G_{\omega} (z,w)$ is
 the kernel of $\Delta_{\omega}^{-1}$ on the orthogonal complement
 of the constant functions.
Thus, in the compact case,  $G_{\omega}$ is defined  by two
conditions: \begin{enumerate}

 \item  $\Delta_{\omega} G_{\omega} (z,w) = \delta_z(w) - \frac{1}{A},$ where $A = \int_{\CP^1} \omega$. That is, $G_{\omega}(z,w)$ is a (singular)
 $\omega$-subharmonic function. In our case $A = 1$.

 \item
$\int_{\CP^1} G_{\omega} (z,w) \omega = 0$.

\end{enumerate}

We denote by  $\{\phi_j\}_{j = 0}^{\infty}$  an orthonormal basis
of eigenfunctions of $\Delta_{\omega}$ in $L^2(\CP^1, \omega)$,
with $\phi_0 = \frac{1}{\sqrt{A}}$ and with $  \Delta \phi_j =
\lambda_j \phi_j$ with $0 = \lambda_0 > \lambda_1 \geq \lambda_2
\downarrow - \infty$. Then $G_{\omega}$ has the eigenfunction
expansion,
\begin{equation} \label{EFNEXP} G_{\omega}(z,w) = \sum_{j = 1}^{\infty}
\frac{\phi_j(z) \phi_j(w)}{\lambda_j}. \end{equation} The
singularity expansion near the diagonal is then   a standard fact
which follows from the Hadamard-Riesz parametrix method (see
\cite{HoIII}, Section 17.4).

We now consider general smooth  $(1,1)$ form $\omega_h$. When
$\omega_h$ fails to be \kahler, we introduce a \kahler metric
$\omega$ in the same cohomology class as $\omega_h$. Since
$\int_{\CP^1} \omega = \int_{\CP^1} \omega_h$,  the $\ddbar$ Lemma
implies that there exists a relative \kahler potential $\phi_{h
g}$ such that $ \omega_h - \omega= dd^c \phi_{h g}. $  By
definition (\ref{Green's potential}), the relative potentials are
given by
\begin{equation} \label{DDCU} dd^c U^{\omega}_h = \omega - \omega_h,
\;\;\; dd^c U^{\omega_h}_{\omega} = \omega_h - \omega.
\end{equation}
It follows that  $U^{\omega}_h = - U^{\omega_h}_{\omega} + A_{gh}$
for a constant $A_{gh}$, and from  $\int U^{\omega_h}_{\omega}
\omega = 0$ we have
\begin{equation} \label{UGH} U^{\omega}_h = -
U^{\omega_h}_{\omega} + \int U^{\omega_h}_{\omega} \omega_h.
\end{equation}  By integrating both sides
against $\omega_h$ we also have $\int U^{\omega}_h \omega = \int
U_{\omega}^{\omega_h} \omega_h. $

 We then
claim that
\begin{equation} \label{GOMGH} G_h(z,w) -  G_{\omega}(z,w) =  U^{\omega}_h (z) + U^{\omega}_h (w)
- \int U^{\omega}_h \omega.
\end{equation}
Since  the  relative potential $U^{\omega}_h$ is  a solution of
the elliptic equation (\ref{DDCU}),  and the left side is
$C^{\infty}$, it follows that $U^{\omega}_h \in C^{\infty}$. Hence
(\ref{GOMGH}) implies the regularity result for any smooth $h$.

To conclude the proof, we need to prove the identity
(\ref{GOMGH}). We observe that the  $dd^c$ derivatives of both
sides of (\ref{GOMGH}) in either $z$ or $w$ agree, both  equaling
$\omega - \omega_h$. Hence, there exists a unique constant
$C_{gh}$ such that
\begin{equation} \label{GOMGHa} G_h(z,w) -  G_{\omega}(z,w) =  U^{\omega}_h (z) + U^{\omega}_h (w) + C_{gh}.
\end{equation}
To determine $C_{gh}$ we integrate both sides of (\ref{GOMGHa})
against $\omega_h(z) \boxtimes \omega(w)$ and use that $\int
G_{\omega} \omega = 0 = \int G_h \omega_h.$ Hence, $$C_{gh} = - (
\int U^{\omega}_h  \omega_h + \int U^{\omega}_h  \omega ) = - \int
U^{\omega}_h  \omega , $$ since $\int U^{\omega}_h  \omega_h = 0$.
This implies
 (\ref{GOMGH}).

\end{proof}

\begin{cor} \label{UUB} With $C_G$ as in Proposition
    \ref{prop-grub},
    for any $\mu \in \mcal(\CP^1)$, $\sup_z U^{\mu}_h
\leq C_G$ \end{cor}

\begin{proof}

This follows from the fact that $U^{\mu}_h(z) = \int_{\CP^1}
G_h(z,w) d\mu(w) \leq C_G, $ as $\int d\mu = 1$.
\end{proof}

\subsection{\label{GE} Green's  energy}
From \eqref{GEN},
we have for the Green's energy with respect to $\omega_h$
\begin{equation} \label{GEDEFN}  \ecal_{h}(\mu) =
    \int_{\CP^1} U^{\mu}_{h}(z) d\mu(z)
=\int_{\CP^1} U^{\mu}_{h} (z) (dd^c U^{\mu}_{h} +
\omega_h)
=\int_{\CP^1} U^{\mu}_{h} (z) dd^c U^{\mu}_{h} ,
\end{equation} where we used (\ref{DGCAL}) in the second equality and the fact that
$\int U^{\mu}_{h} \omega_h = 0$ in the  last equation.

 In the next result, we outline a  proof of the  convexity of the energy
 functional for general
smooth Hermitian metrics. It is used in the proof of the convexity
of the
rate function in  Lemma \ref{CONVEX}. Convexity of the
energy  is well-known in weighted potential theory: It is
  proved in Lemma 1.8 of \cite{ST} that $- \Sigma(\mu) \geq 0$ in
  the case where
  case where $\mu = \mu_1 - \mu_2$ is a signed Borel measure with
  compact support, where $\mu(\C) = 0$ and each of $\mu_1, \mu_2$
  satisfies $- \Sigma(\mu_j) < \infty$. A different proof is given
  in \cite{BG}, Property 2.1(4) and another in Proposition 5.5 of
  \cite{BB}. We give a somewhat different proof in our setting of
  $\CP^1$.

We define the energy form  on $\mcal(\CP^1)$ by
\begin{equation} \label{ENERGYNORM} \langle \mu, \nu \rangle_{\omega} := \int_{\CP^1} G_{\omega}(z,w) d\mu(z) d\nu(w) =  \int_{\CP^1}
U^{\mu}_{\omega} d\nu = \int_{\CP^1} U^{\nu}_{\omega} d\mu.
\end{equation}
As in \cite{C}, we denote the probability measures of finite
energy $||\mu||^2_{\omega} < \infty$  by $\ecal^+(\CP^1) \subset
\mcal(\CP^1)$.

 \begin{prop}\label{CONVEXITY} For any smooth Hermitian metric on $\ocal(1)$,  $- \ecal_h$ is a convex
     functional on $\mcal(\CP^1)$.

\end{prop}

\begin{proof}

We first prove convexity when $\omega$ is a \kahler metric.
  \begin{lem} \label{NONPOS}
When $\omega$ is a \kahler metric, the energy form $\langle \mu,
\nu \rangle_{\omega}$ is negative semi-definite   on signed
measures of finite energy. The unique measures of energy zero are
multiples of $\omega$.
\end{lem}
\begin{proof}

 From the eigenfunction expansion (\ref{EFNEXP}), it follows
that
 \begin{equation} \label{GRINNPROD} \langle\mu, \nu\rangle_{\omega}
     = \int_{\CP^1\times \CP^1} G_{\omega}(z,w)
 d\mu(z) d\nu(z) =  \sum_{j
 = 1}^{\infty} \frac{\mu(\phi_j) \nu (\phi_j)}{\lambda_j}.
 \end{equation}
It is clear that for any signed measure $\mu$,  $\langle\mu,
\mu\rangle_{\omega} \leq 0$ with equality if and only if $\mu(\phi_j) =
0$ for all $j = 1, 2, \dots.$ The constant term has been removed
from the sum, so this case of equality is only possible if and
only if  $\mu = C \omega$ for some constant $C$.

\end{proof}

We then let $h$ be a general smooth metric. The following lemma
follows  immediately from the identity (\ref{GOMGH}).

\begin{lem} \label{ENERGYFORMS}Let $h$ be any smooth Hermitian metric, and let
$\omega$ be a \kahler form with $\int_{\CP^1} \omega =
\int_{\CP^1} \omega_h. $ Then, their energy forms are related by
$$\langle \mu, \nu\rangle_{\omega_h} = \langle \mu,
\nu\rangle_{\omega} + \nu(\CP^1) \int U_{\omega}^{\omega_h} d \mu
+ \mu(\CP^1) \int U_{\omega}^{\omega_h} d\nu.
$$

\end{lem}

It follows that on $\mcal(\CP^1)$, $- \ecal_h$ is a convex
function.

\end{proof}

\subsection{\label{GFL2} Green's function and $L^2$ norms}

\begin{lem} \label{GREENNORM}
    Let $G_h$ be the Green's function relative to  $\omega_h$.  Then,
$$(i)\;\; e^{ \int_{\CP^1} G_{h}(z,w) dd^c \log
||s(w)||_{h^N}^2 } = ||s||_{h^N}^2(z) e^{ - \int_{\CP^1} \log
||s||_{h^N}^2(z) \omega_h}, $$ and
$$(ii)\; \begin{array}{lll} \int_{\CP^1 \times \CP^1 \backslash D} G_{h}(z,w) dd^c \log
||s(z)||_{h^N}^2  \boxtimes dd^c \log ||s(w)||_{h^N}^2  & = &
\int_{\CP^1 \times \CP^1 \backslash D} G_{h}(z,w) Z_s \boxtimes
Z_s.
\end{array} $$
\end{lem}

\begin{proof}

    For the first formula, we  take
    the logarithm of both sides, and  integrate $dd^c$ by parts in
the integral $ \int_{\CP^1} G_{h}(z,w) dd^c \log ||s(w)||_{h^N}^2
$. It is possible since both factors are global currents on
$\CP^1$.   The resulting integral equals $\log ||s(z)||_{h^N}^2 -
\int_{\CP^1} \log ||s(z)||_{h^N}^2 \omega_h. $

For (ii) we write $\frac{1}{N} dd^c \log ||s(z)||_{h^N}^2 = Z_s -
\omega_h$. Then

$$\begin{array}{lll} (ii) & = &
\int_{\CP^1 \times \CP^1 \backslash D} G_{h}(z,w) (Z_s -
\omega_h) \boxtimes (Z_s - \omega_h) \\ && \\
& =& \int_{\CP^1 \times \CP^1 \backslash D} G_{h}(z,w) Z_s
\boxtimes Z_s  \\ && \\
&& - 2 \int_{\CP^1 \times \CP^1 \backslash D} G_{h}(z,w) Z_s
\boxtimes  \omega_h + \int_{\CP^1 \times \CP^1 \backslash D}
G_{h}(z,w)  \omega_h \boxtimes  \omega_h\\ && \\ & =& \int_{\CP^1
\times \CP^1 \backslash D} G_{h}(z,w) Z_s \boxtimes Z_s ,
\end{array}$$
since $\int G_h(z,w) \omega_h = 0$ when integrating in either $z$
or $w$. By Proposition \ref{prop-grub}, $G_h \in L^1(\CP^1,
\omega_h)$;  the integral over $\CP^1 \times \CP^1 \backslash D$
in the last terms is the same as over $\CP^1 \times \CP^1$.
\end{proof}

\begin{cor} \label{GREENENERGY} We have:
$$\begin{array}{lll} e^{\int_{\CP^1 \times \CP^1 \backslash D} G_{h}(z,w)
    dd^c \log
||s_{\zeta}(z)||_{h^N}^2  \boxtimes dd^c \log
||s_{\zeta}(w)||_{h^N}^2 } & = & e^{\sum_{i \not= j} G_h(\zeta_i,
\zeta_j)}
\end{array}$$
\end{cor}

\subsection{\label{JPCGF} Completion of proof of Proposition \ref{FSVOLZETA2intro}}

We now  complete the proof of Proposition \ref{FSVOLZETA2intro},
which was started in \S \ref{DETG0}. The purpose of this section
is to convert the local expression \eqref{eq-030209b} (see also
\eqref{eq-030209d}) for the joint probability current into a
global invariant expression. We prove:

\begin{lem} \label{JPDINV}
    Let $h = e^{- \phi}$ be a
    smooth Hermitian metric on $\ocal(1)$, and let
    $\omega_h, G_h$ be as above.  Let
    $s_{\zeta}(z) = \prod_{j = 1}^N (z - \zeta_j) e^N$.
    Then, the  joint probability current is given by:
$$ \begin{array}{lll}  \frac{|\det \acal_N(h, \nu) |^2 |\Delta(\zeta_1, \dots,
\zeta_N)|^2\prod_{j = 1}^N d^2 \zeta_j }{\left(\int_{\CP^1}
\prod_{j = 1}^N |(z - \zeta_j)|^2 e^{-  N \phi} d\nu(z)
\right)^{N+1}} & = &   \frac{\exp \left(\frac{1}{2} \sum_{i \not=
j} G_{h}(\zeta_i, \zeta_j) \right) }{\left(\int_{\CP^1} e^{
\int_{\CP^1} G_{h}(z,w) d\mu_{\zeta}(w)} d\nu(z) \right)^{N+1}}
\left(\prod_{j = 1}^N e^{- 2  \phi(\zeta_j)} d^2
\zeta_j \right) \\ && \\
& \times & |\det \acal_N(h, \nu) |^2  e^{(-\frac{1}{2}N (N - 1) +N
(N + 1)) E(h)} ,
\end{array}$$
where $E(h)$ is defined in (\ref{EH}) and $\acal_N$ is defined in
(\ref{ACALH}).  Moreover, $\prod_{j = 1}^N e^{- 2  \phi(\zeta_j)}
d^2 \zeta_j$ extends to a global smooth $(N, N)$ form
$\kappa_N$ on
$(\CP^1)^N$.
\end{lem}

\begin{proof} We first claim that
\begin{equation}\label{DELTASQUARE}  |\Delta(\zeta)|^2  =  \exp \left( \sum_{i < j} G_{h}(\zeta_i,
\zeta_j) \right) \exp\left(   (N - 1) \sum_j \phi(\zeta_j) -
 \frac{1}{2} (N - 1) N  E(h) \right).
\end{equation}

Indeed, by   Lemma \ref{GHFORM},
$$  2 \log |z - w| = G_{h}(z,w) +  \phi(z) + \phi(w) -  E(h).  $$
We note that $\log |\Delta(\zeta)|^2 = 2 \sum_{i < j} \log
|\zeta_i - \zeta_j|$ and that
\begin{equation}\begin{array}{lll} 2  \sum_{i < j} \log |\zeta_i -
\zeta_j| & = &  \sum_{i < j} G_{h}( \zeta_i,  \zeta_j) +
 \sum_{i
< j} ( \phi(\zeta_i) + \phi(\zeta_j)) -   E(h)) \\
&& \\
& = &  \sum_{i < j}G_{h}( \zeta_i,  \zeta_j)  +  (N - 1)
\sum_j \phi(\zeta_j) -  \frac{1}{2} (N - 1) N  E(h)\\&& \\
& = &  \sum_{i < j}G_{h}( \zeta_i,  \zeta_j)  +  N (N - 1) \int
\phi d\mu_{\zeta} - \frac{1}{2} (N - 1) N  E(h).
\end{array}
\end{equation}

We then convert the denominator into the Green's function
expression by the identities \begin{equation} \label{GRID}
\begin{array}{lll}\int_{\CP^1} \prod_{j = 1}^N |(z - \zeta_j)|^2
e^{-  N \phi} d\nu(z) & = & \int_{\CP^1} ||s_{\zeta}(z)||^2_{h^N}
d\nu(z)\\ && \\
& = & \left( \int_{\CP^1} e^{\int_{\CP^1} G_{h}(z,w) dd^c \log
||s_{\zeta}(w)||_{h^N}^2} d\nu \right) e^{  \int_{\CP^1} \log
||s_{\zeta}||_{h^N}^2(z) \omega_h} \\ && \\
& = & \left( \int_{\CP^1} e^{N \int_{\CP^1} G_{h}(z,w)
d\mu_{\zeta}(w)} d\nu \right) e^{  \int_{\CP^1} \log
||s_{\zeta}||_{h^N}^2(z) \omega_h}
\end{array} \end{equation} by   Lemmas \ref{LOGMUZ} and \ref{GREENNORM} (i).
  Further, by Lemma \ref{NORMSUPHI},
$$  \int \log
||s_{\zeta} (w)||_{h^N}^2  \omega_h = N (\int \phi d\mu_{\zeta} -
E(h)). $$

We now raise the denominator (\ref{GRID}) to the power $- (N + 1)$
and multiply by (\ref{DELTASQUARE}) to obtain the Green's
expression
$$\frac{\exp \left( \sum_{i < j}
G_{h}(\zeta_i, \zeta_j) \right) }{\left(\int_{\CP^1} e^{
\int_{\CP^1} G_{h}(z,w) d\mu_{\zeta}(w) } d\nu(z) \right)^{N+1}}
$$ multiplied by the exponential of
$$ \begin{array}{l}  (N - 1) N \int \phi d\mu_{\zeta}-
 \frac{1}{2} (N - 1) N  E(h)
 -  N (N + 1)  (\int \phi d\mu_{\zeta} - E(h)).
\end{array}$$
We  note the   cancellation in the $N^2$ term of $\int \phi
d\mu_{\zeta}$, leaving $- 2 N \int \phi d\mu_{\zeta} = -2 \sum_j
\phi(\zeta_j)$ This gives the stated result. The last statement
follows from Lemma \ref{INVAR}.

\end{proof}

\subsection{The approximate rate function $I_N$}

Lemma \ref{JPDINV} expresses the joint probability  $\vec
K_n^N(\zeta_1, \dots, \zeta_N)$ as a geometric $(N + 1, N + 1)$
form on configuration space. In order to extract a rate function,
we further express it as a functional of the measures
$\mu_{\zeta}$. We introduce the following functionals.

\begin{defin}  \label{DEFIN}  Let $\zeta \in (\CP^1)^{(N)}$ and let $\mu_{\zeta}$ be as in (\ref{ZN}).
Let $D=\{(z,z): z\in \CP^1\}$ be the diagonal. Put:
$$ \left\{ \begin{array}{l} \ecal^h_N(\mu_{\zeta}) = \int_{\CP^1 \times \CP^1 \backslash D}
G_{h}(z,w) d\mu_{\zeta}(z) d\mu_{\zeta}(w), \\ \\
 J_N^{h, \nu} (\mu_{\zeta}) = \log ||e^{U_h^{\mu_{\zeta}}}||_{L^N(\nu)}
 \end{array}
\right.$$
\end{defin}

\begin{lem} \label{APPROXRATE} We have
$$\vec K_n^N(\zeta_1, \dots, \zeta_N) = \frac{1}{\hat{Z}_N(h)} e^{- N^2 \left( -\frac{1}{2}
\ecal^h_N(\mu_{\zeta}) + \frac{N+1}{N}  J_N^{h, \nu} (\mu_{\zeta})\right)}
\kappa_N
$$
\end{lem}

\begin{proof} We are simply rewriting
\begin{equation} \label{IN}   \frac{\exp \left(\frac{1}{2} \sum_{i \not= j}
G_{h}(\zeta_i, \zeta_j) \right) }{\left(\int_{\CP^1} e^{
\int_{\CP^1} G_{h}(z,w) d\mu_{\zeta}(w)} d\nu(z) \right)^{N+1}} =
e^{- N^2 I_N(\mu_{\zeta})}, \end{equation}  on the right side of
Lemma \ref{JPDINV} and leaving the other factors as they are.
Then,
\begin{equation} \label{INV2}
    \begin{array}{lll}
    I_N(\mu_{\zeta})
& = &  - \frac{1}{N^2}\sum_{i \not= } \frac{1}{2} G_{h}(\zeta_i,
\zeta_j) + \frac{N+1}{N^2} \log \left(\int_{\CP^1} e^{N
\int_{\CP^1} G_{h}(z,w) d\mu_{\zeta} } d\nu(z) \right)\nonumber \\
&& \quad
\quad \\
 & = &
 - \frac{1}{N^2}  \frac{1}{2} \int_{\CP^1 \times \CP^1
\backslash D} G_{h}(z,w) d\mu_{\zeta}(z) d\mu_{\zeta}(w) +
\frac{N+1}{N^2} \log \left(\int_{\CP^1} e^{N U^{\mu_{\zeta}}_h(z)
} d\nu(z) \right)  \nonumber  \\ && \\ & = & - \frac{1}{N^2}
\left(
 - \frac{1}{2}  \ecal_N^h(\mu_{\zeta}) +  \frac{N(N + 1)}{N^2}   J_N^{h, \nu} (\mu_{\zeta})\right).
\end{array}
\end{equation}
  \end{proof}

\section{\label{WEM} Weighted equilibrium measures }

In this section, we define the notion of weighted equilibrium
measure $\nu_{h, K}$ of a non-polar compact set $K$ with respect
to a Hermitian metric $h$ and prove that it is unique. In fact,
there are two characterizations of $\nu_{h, K}$:
\begin{itemize}
\item[(i)] $\nu_{h,K}$ is the
minimizer of the Green's energy functional among measures
supported on $K$.
\item[(ii)] The potential of $\nu_{h,K}$
is the maximal
$\omega_h$-subharmonic function of $K$.
\end{itemize}
We will need both
characterizations in order to prove that the unique minimizer of
the
function $I^{h, K}$ of (\ref{IGREEN}) is $\nu_{h, K}$.
The problem is that $I^{h, K}$ differs significantly from the
Green's energy on $K$ and it is not obvious that they have the
same minimizer.

In the classical case of weighted potential theory on $\C$, the
equivalence of the two definitions is proved in \cite{ST},
especially in the appendix by T. Bloom. Their framework of
admissible weights on $\C$ does not quite  apply directly to the
present setting of smooth Hermitian metrics on $\ocal(1)$ and
potential theory on $\CP^1$. The second definition (ii) is assumed
in work on potential theory on \kahler manifolds, e.g. as in
\cite{GZ}, Definition 4.1.  Only recently in \cite{BB} have
equilibrium measures been considered in terms of energy
minimization. As a result, there is no simple reference for the
facts we need, although their proofs are often small modifications
of known proofs in the weighted case on $\C$. In that event, we
only sketch the proof and refer the reader to the literature.

\subsection{Equilibrium measures as energy minimizers} We now
justify the first definition (i) by showing  that there exists a
unique energy minimizer (or maximizer, depending on the sign of
the energy functional) among measures supported on a non-polar set
$K$. We further prove  that weighted equilibrium measures $\nu_{h,
K}$ are unique and are supported on $K$ (Proposition \ref{EQMES}).
We recall that $K$ is  a polar set if $\ecal_h(\mu)=-\infty$
for every finite non-zero Borel measure $\mu$
supported in $K$. In particular, a set satisfying
\eqref{REGULAR} is non-polar.

Thus, we fix a  compact non-polar subset $K \subset \CP^1$ and
consider the  restriction of the energy functional
$\ecal_{\omega}: \mcal(K) \to \R$ of probability measures
supported on $K$.

 \begin{prop}\label{EQMES}  If $K \subset \CP^1$ is non-polar, then $\ecal_{\omega}$ is bounded above on $\mcal(K)$. It has
 a unique maximum $\nu_{K,\omega} \in \mcal(K)$.

\end{prop}

We denote its potential, the {\it weighted equilibrium potential},
by

\begin{equation} \label{EQPOT} U^{\nu_{K,\omega} }_{\omega}(z) =
 \int G_{\omega}(z,w) d\nu_{K, \omega}(w). \end{equation}

\begin{proof} We begin by sketching the proof in the case where
$\omega = \omega_h$ is a \kahler metric.  In this case, the proof
follows the standard lines of \cite{Ran} (Theorem 3.3.2 and
Theorem 3.7.6) or
 \cite{ST}, Theorem 1.3 and particularly Theorem 5.10. Existence follows from the upper
 semi-continuity of $\ecal_{\omega}$, which holds exactly as in
 the local weighted case.

 Uniqueness by the method  of \cite{ST}, Theorem I.1.3 or  Theorem II.5.6
 uses the  non-positivity of the  weighted logarithmic energy norm (Lemma I.1.8 of \cite{ST}) or of
 the Green's energy norm (Theorem II.5.6). This argument applies directly to $\ecal_{\omega}$ when $\omega$ is
 \kahler: one assumes for purposes of contradiction that
there exist two energy maximizers $\mu, \nu$ of mass one. Then it
follows by the argument of Theorem I. 1.3 (b) of \cite{ST} that
$||\mu - \nu||_{\omega}^2 = 0$; so $\mu - \nu = C \omega$. But
integration over $\CP^1$ shows that $C = 0$, proving that $\mu =
\nu$.

We then consider a general smooth Hermitian metric $h$.
Since
$\ecal_h$ is bounded above and $\mcal(K)$ is closed and hence compact,
there exist
measures in $\mcal(K)$ which maximize the energy $\ecal_h$. We now
prove uniqueness:

 \begin{lem}\label{EQMESKapp0}
If $K \subset \CP^1$ is non-polar, and $h$ is any smooth metric,
 then $\ecal_{h}$  has
 a unique maximizer $\nu_{K,h} \in \mcal(K)$.
\end{lem}

\begin{proof}
We put
\begin{equation} \label{VMAX}V_{\omega_h}(K) = \max
\{\ecal_{h}(\mu): \mu \in \mcal(K)\}<\infty.  \end{equation}
To prove uniqueness, we observe that,   for any signed measure
$\mu - \nu$ given by a difference of two elements of
$\mcal(\CP^1)$, hence satisfying $\int_{\CP^1} d (\mu - \nu) = 0$,
we have
\begin{equation}\label{CURIOUS} ||\mu - \nu||_{\omega_h}^2 = ||\mu - \nu||_{\omega}^2. \end{equation}    Indeed,
\begin{equation} \begin{array}{l}  ||\mu - \nu||_{\omega_h}^2  =  \int_{\CP^1 \times \CP^1} G_h(z,w) d(\mu - \nu)
\otimes d (\mu - \nu) \\ \\  =  \int_{\CP^1 \times \CP^1}
(G_{\omega}(z,w) + U^{\omega}_h (z) + U^{\omega}_h (w) - \int
U^{\omega}_h \omega) ) d(\mu - \nu)
\otimes d (\mu - \nu) \\ \\
=  \int_{\CP^1 \times \CP^1} G_{\omega}(z,w) d(\mu - \nu) \otimes
d (\mu - \nu) = ||\mu - \nu||_{\omega}^2.  \end{array}
\end{equation}

 Hence, the
energy form is negative semi-definite on the subspace of signed
measures  $\mu - \nu$ where $\mu, \nu$ are positive and of the
same mass.

 Suppose that
$\mu, \nu \in \mcal(K)$ and that both are maximizers of $\ecal_h$
on $\mcal(K)$. Then $\int_{\CP^1} d (\mu - \nu) = 0$ and hence
$||\mu - \nu||_{\omega}^2 \leq 0$. Equality holds if and only if
$\mu - \nu = C \omega$ for some $C$, and the fact that $\int d
(\mu - \nu) = 0$ implies $C = 0$. But
\begin{equation} ||\frac{1}{2}(\mu + \nu))||_h^2  +  ||\frac{1}{2}(\mu - \nu)||_{\omega_h}^2 = \frac{1}{2}(
\ecal_{h}(\mu) + \ecal_h(\nu)) = V_h(K). \end{equation} Since
$\ecal_{h}(\sigma) \leq V_{\omega_h}(K)$ for any $\sigma \in
\mcal(\CP^1)$, it follows that $||\mu - \nu||_{\omega}^2 = 0$ and
hence $\mu = \nu.$ This completes the proof of uniqueness.

\end{proof}

This completes the proof of Proposition \ref{EQMES}. \end{proof}

\begin{defin}  The weighted capacity of $K$ with
respect to $\omega$ is defined by
\begin{equation}
    \label{CAPKOMEGA}
     \mbox{Cap}_{\omega}(K) = e^{ \sup
\{\ecal_{\omega}(\mu): \mu \in \mcal(K) \}} = e^{
\ecal_{\omega}(\nu_{K, \omega})}.
\end{equation}
\end{defin}

\subsection{\label{SHEM}Equilibrium measure and subharmonic envelopes}
We now discuss the second characterization (ii) of equilibrium
measures (see ithe beginning of
\S \ref{WEM})  and prove that it is equivalent to the first.

 Given a closed real $(1,1)$
form $\omega$ (not necessarily a \kahler form), and a compact
subset $K \subset \CP^1$, define the global extremal function
$V_{K, \omega}^*$ as the upper semi-continuous regularization of
\begin{equation} \label{VKOMEGA} V_{K, \omega}(z) : = \sup\{u(z): u \in SH(\CP^1,
\omega, K)\}, \end{equation} where  \begin{equation}
\label{SHOMEGAK} SH(\CP^1, \omega, K) : = \{u \in SH(\CP^1,
\omega): u \leq 0 \;\; \mbox{on}\;\; K\}. \end{equation}
(See \eqref{SHH} for the definition of $SH(\CP^1,\omega)$.)

The important properties of $V_{K, \omega}^*$ and $\nu_{K,
\omega}$ are the following, a special case of Theorem 4.2 of
\cite{GZ}:

\begin{theo}\label{GZ} Let $K \subset \CP^1$ be a Borel set. If $K$ is
non-polar, then $V_{K, \omega}^* \in SH(\CP^1, \omega)$ and
satisfies:
\begin{enumerate}

\item $\nu_{K, \omega} = 0$ on $\CP^1 \backslash \overline{K}$.

\item $V_{K, \omega}^* = 0$ on supp $\nu_{K, \omega}$ and in the
interior of $K$;

\item  $\int_{\overline{K}} \nu_{K, \omega} = \int \omega \; (=
1)$.

\end{enumerate}

\end{theo}

The following Proposition relates  $V_{K, \omega}^*$ to  the
potential (\ref{EQPOT}) of the equilibrium measure of Proposition
\ref{EQMES}.

 \begin{prop}\label{WvsU}  Let $K \subset \CP^1$ be a  non-polar  compact subset and
let $\omega$ be a smooth $(1,1)$ form with $\int_{\CP^1} \omega =
1$. Then,
\begin{equation} \label{EQMEASH} \nu_{K, \omega} = dd^c V_{K,
\omega}^* + \omega. \end{equation} Moreover,
 \begin{equation} \label{UOMEGAKVOMEGAKa} U^{\nu_{K, \omega}}_{\omega} =
V_{K, \omega}^* - \int_{\CP^1} V_{K, \omega}^* \omega.
\end{equation}

\end{prop}
In particular, with $F_{K, \omega}=\int_{\CP^1} V^*_{K,\omega} \omega$,
we have that
$U_\omega^{\nu_{K,\omega}}=-F_{K, \omega}$ on the support of
$\nu_{K,\omega}$.

 \begin{proof} We only sketch the proof, which is barely different
 from the case of admissible potential theory on $\C$ \cite{ST}.

  The  second statement (\ref{UOMEGAKVOMEGAKa}) implies the first. It shows that both
  $U^{\nu_{K,\omega} }_{\omega}$ and $V_{K,
 \omega}^*$
 belong to $SH(\CP^1, \omega)$ and  are  potentials for
 $\nu_{K, \omega}$, i.e.
 \begin{equation}\label{DDCFORM}  dd^c U^{\nu_{K,\omega} }_{\omega}  = \nu_{K,
 \omega} - \omega =  dd^c V_{K, \omega}^*. \end{equation}
The potentials must differ by a constant, which is determined by
 integrating with respect to $\omega$ and using that $\int
 U^{\mu}_{\omega} \omega = 0$ for any $\mu$. The proof is
 essentially the same as in the classical unweighted case (see
 Lemma 2.4 of Appendix B.2 of \cite{ST}).

It therefore suffices to prove  (\ref{UOMEGAKVOMEGAKa}).
 The proof in the case of weighted potential theory on $\C$ is given in   Theorem I.4.1 of \cite{ST}, as sharpened in
 Appendix B, Lemma 2.4 of \cite{ST}. The main ingredients are the so-called principle of domination (see \cite{ST}, I.3),
 and the Frostman type theorem that $U^{\nu_{K, \omega}}_{\omega} \geq F_{\omega}$ q.e. on $K$ and
 $U^{\nu_{K, \omega}}_{\omega} \leq F_{K, \omega}$ on
supp $\nu_{K, \omega}$, hence
 $U^{\nu_{K, \omega}}_{\omega} =  F_{K, \omega}$ on
supp $\nu_{K, \omega}$ (see \cite{ST}, Theorem I.1.3 (d)-(f)).
\end{proof}
We may view the energy function as a function on $SH(\CP^1,
\omega)$ rather than on $\mcal(\CP^1)$:
\begin{equation}  \ecal_{\omega}(\psi) : =
 \int_{\CP^1} \psi  (dd^c \psi + \omega). \end{equation}
This definition is slightly more general than the preceding one
since $\int U^{\mu}_{\omega} \omega = 0$ but  $\psi$
 is not assumed to be so normalized.

\begin{rem} We note that $V_{K, \omega}^* \in SH(\CP^1, \omega, K)$ and that
$U^{\nu_{K, \omega}}_{\omega} \in SH(\CP^1, \omega, K)$. In
classical weighted potential theory, both $U^{\nu_{K,
\omega}}_{\omega}$ and $V_{K, \omega}^* $ are defined slightly
differently, and their difference $F_{K,\omega}$
is known as the {\it Robin
constant}.
\end{rem}
\begin{cor}\label{MAXENERGY} We have
\begin{enumerate}
\item $U^{\nu_{K, \omega}}_{\omega}$ maximizes $\ecal_{\omega}$
among all elements of $SH(\CP^1, \omega,K)$.

\item $V^*_{K, \omega}$ maximizes $\ecal_{\omega}$ among all
elements of $SH(\CP^1, \omega, K)$.

\end{enumerate}
\end{cor}

\subsection{\label{IRREG} Thin points, regular points and capacity}

A set $E$ is said to be {\it thin} at $x_0$ iff either of the
following occur:

\begin{itemize}

\item $x_0$ is not a limit point of $E$;

\item  there exists a disc $D_{\epsilon}(x_0)$ and $\epsilon
> 0$ and a potential $U^{\mu}$ so that $U^{\mu}(x_0) > - \infty$,
so that $U^{\mu}(x) \leq U^{\mu}(x_0) - \eta$ for all $x \in E
\cap D_{\epsilon}(x_0) \backslash \{x_0\}$.

\end{itemize}
We refer to \cite{Lan}, Definition Ch. V \S 3 (5.3.1) or \cite{D}.
A point $x_0$ is called irregular for $E$ if $E$ is thin an $x_0$.
Thus, our assumption on $K$ is that it is non-thin at all of its
points. A subset of a thin set at $x_0$ is also thin at $x_0$ and
the union of two thin sets at $x_0$ is thin there.

We further recall:
\begin{lem} \label{POS} (see \cite{ST}, Corollary 6.11,
    or the Corollary to Theorem 3.7 of \cite{Lan}, Ch. III \S 2)
    If $S \subset \C$ is compact and of positive capacity, then
there exists a positive, finite  measure $\nu$, with support
included in $S$,
so that $U^{\nu} \in C(\CP^1)$. \end{lem}
 The idea is
that $S$ contains a regular subset of positive capacity.

\section{\label{Rate} Rate function and equilibrium measure}

We continue to  fix a pair $(h, \nu)$ where $h$ is a smooth
Hermitian metric on $\ocal(1)$ and where $\nu$ is a measure
satisfying (\ref{BM}). The purpose of this section is to prove
that the rate function
(\ref{eq-ofer1}) of the LDP of Theorem
\ref{g0} is a {\it good rate function} and also to prove that its
unique minimizer is the equilibrium measure for $(h, K)$. That is,
we prove:
\begin{prop} \label{GOODRATEF} The function $I^{h, K}$ of
(\ref{IGREEN}) has the following properties:

\begin{enumerate}
\item It is a lower-semicontinuous functional. \item It is convex.
\item Its unique minimizer is the equilibrium measure $\nu_{h,
K}$. \item Its minimum value equals $\frac{1}{2}  \log
\mbox{Cap}_{\omega_h} (K).$
\end{enumerate}

\end{prop}

We begin with the following elementary consequence of
Proposition \ref{prop-grub}.
\begin{lem}
    \label{lem-ue}
    For each $z\in \CP^1$, the function
    $\mu\to U_h^\mu(z)$ is an upper  semi-continuous
    function from $\mcal(\CP^1)$ to $\R\cup \{-\infty\}$. Further,
    so is the function $\mu\to\ecal_h(\mu)$.
\end{lem}
\begin{proof}
    Fix $M\in \R$ and define $G_h^M(z,w)=G_h(z,w)\vee (-M)$.
    By Proposition \ref{prop-grub}, $G_h^M$ is continuous on
    $\CP^1\times \CP^1$. Set
    \begin{equation}
        \label{eq-180209a}
        U^{\mu,M}_h(z)=\int_{\CP^1}G_h^M(z,w)d\mu(w)\,,\quad
    \ecal_h^M(\mu)=\int_{\CP^1\times\CP^1} G_h^M(z,w)d\mu(z)d\mu(w)\,.
\end{equation}
    For fixed $z$, it follows that $\mu\to U^{\mu,M}_h(z)$ and
    $\mu\to \ecal_h^M(\mu)$ are continuous on $\mcal(\CP^1)$.
    Since $U^{\mu}_h(z)=\inf_M U^{\mu,M}_h(z)$ and
    $\ecal_h(\mu)=\inf_M \ecal_h^M(\mu)$, the claimed
    upper semi-continuity follows.
\end{proof}

We next have the following.
\begin{lem} \label{LSCJG}
\begin{enumerate}
    \item The function $J^{h, K}(\mu) =  \sup_{z\in K} U^{\mu}_h(z)$
        is upper
semi-continuous. \item Assume that all points of $K$ are regular.
Then $J^{h, K}(\mu) $ is also lower semi-continuous.

\end{enumerate}
\end{lem}
\begin{proof}

\noindent{\bf (i) Upper semi-continuity}

\medskip

    We begin by proving the upper semi-continuity.  Let
    $\mu_n\to \mu^*$ weakly in $\mcal(\CP^1)$. Fix $M\in \R$  and recall
    that $G_h^M(\cdot,\cdot)$ is continuous on
    $\CP^1\times \CP^1$. Therefore, the map
    $(z,\mu)\to U_h^{\mu,M}(z)$ is continuous. Because
    $K$ is compact, $\mu\mapsto \sup_{z\in K} U_h^{\mu,M}(z)$ is
    therefore continuous. Thus,
    $$ J^{h,K}(\mu_n)=\sup_{z\in K} U_h^{\mu_n}(z)
\leq
\sup_{z\in K} U_h^{\mu_n,M}(z)\to_{n\to\infty}
\sup_{z\in K}U_h^{\mu^*,M}(z)\,.
$$
Since $U_h^{\mu^*,M}(z)\to_{M\to\infty} U_h^{\mu^*}(z)$
for any $z$ by monotone convergence, we have
$$\sup_{z\in K}U_h^{\mu^*,M}(z)\to_{M\to\infty}
\sup_{z\in K}U_h^{\mu^*}(z)=J^{h,K}(\mu)\,.$$ Combining the last
two displays completes the proof of (1).

\medskip

\noindent{\bf (ii) Lower semi-continuity}

\medskip

 Let  $K$ be a set all of
whose points are regular. Suppose that $\mu_n \to \mu$. We claim
that
$$ \liminf_{n \to \infty} \sup_K U^{\mu_n}_{h}  \geq
\sup_{z\in K} U^{\mu}(z):=a.
$$
(Recall that $a<\infty$.)
For any $\epsilon > 0$ set
$$A_{\epsilon} = \{z \in \CP^1: U_h^{\mu} \geq a - \epsilon\}.  $$
We claim the following
\begin{equation} \label{CLAIM}\;\;
\forall \epsilon > 0,\;\;\mbox{Cap}(A_{\epsilon} \cap K)
> 0. \end{equation}

To prove the claim, let $z^*$ be a point where $U^{\mu}$ attains
its maximum on $K$ (such a point
exists by the upper semicontinuity of $z\mapsto U^\mu(z)$ and the
compactness of $K$).
By assumption, $z^*$ is a regular point. Since
$$U_h^{\mu} <  a - \epsilon ,  \;\; \mbox{on} \; K \backslash
A_{\epsilon}, $$ $K \backslash A_{\epsilon}$ is thin at $z^*$
 (see \cite{Lan} (5.3.2), page 307). Suppose that there exists
$\epsilon_0 > 0$ so that $\mbox{Cap}(A_{\epsilon_0} \cap K) = 0$.
A set of capacity zero is thin at each of its points, see  \cite{Ch},
Corollary
on page 92.
Since
$$K = K \backslash
A_{\epsilon_0} \cup (K \cap A_{\epsilon_0}), $$ and since the union
of two sets thin at $z^*$ is thin at $z^*$,  we see that $K$ is
thin at $z^*$. This contradicts the regularity of $K$ at $z^*$ and
proves the claim.

We now complete the proof of the Proposition. Let ${\bf
1}_{A_{\epsilon} \cap K}$ denote the characteristic function of
$A_{\epsilon} \cap K$. Since $U^{\mu}$ is upper semi-continuous,
$A_{\epsilon} \cap K$ is compact. By (\ref{CLAIM}), it has
positive capacity. It follows by Lemma \ref{POS} there exists a
positive measure $\nu_{\epsilon, \mu, K}$ supported on
$A_{\epsilon} \cap K$  whose potential $U^{\nu_{\epsilon, \mu,
K}}$ is continuous.

We have,
\begin{equation} \begin{array}{lll}  \label{LIMINFORMa}
    \lim_{n \to \infty}
\int_{A_{\epsilon}} U_h^{\mu_n}(z) d \nu_{\epsilon, \mu, K }(z) &
=&  \lim_{n \to \infty} \int U_h^{\nu_{\epsilon, \mu, K}}(z)
d\mu_n(z)\\ && \\&=& \int U_h^{\nu_{\epsilon, \mu, K}}
d\mu(z)\nonumber\\&&\\ &=& \int_{A_{\epsilon}} U_h^{\mu}(z) d
\nu_{\epsilon, \mu, K}(z) \nonumber.
\end{array} \end{equation}
Therefore,
\begin{eqnarray*}
\nu_{\epsilon, \mu, K} (A_{\epsilon} \cap K) \liminf_{n \to
\infty} \sup_K U^{\mu_n}_{h} & \geq &   \liminf_{n \to \infty}
\int_{A_{\epsilon}} U^{\mu_n}_{h}(z) d \nu_{\epsilon, \mu, K}(z)\\
&=& \int_{A_{\epsilon}} U^{\mu}_{h}(z) d\nu_{\epsilon, \mu,
K}(z)
 \geq  (a - \epsilon) \nu_{\epsilon, \mu, K} (A_{\epsilon} \cap K).
 \end{eqnarray*}
Since $\nu_{\epsilon, \mu, K} (A_{\epsilon} \cap K) > 0$ and since
$\epsilon$ is arbitrary, this finishes the proof.
\end{proof}

\begin{rem}
We note that if $d\nu$ is any measure on $K$ whose potential
$U^{\nu}$ is continuous,  then $U^{{\bf 1}_{A_{\epsilon} \cap K}
\nu}$ is automatically  continuous. Indeed, $U^{{\bf
1}_{A_{\epsilon} \cap K} \nu}$ is upper semi-continuous, so we
only need to prove that it is lower semi-continuous. But
$$U^{{\bf 1}_{A_{\epsilon} \cap K} \nu}  = U^{\nu} - U^{\nu - {\bf 1}_{A_{\epsilon} \cap K} \nu}, $$
and the first term on the right is continuous and the second,
being the opposite of a potential, is lower semi-continuous.

\end{rem}

A consequence of Lemma \ref{LSCJG} is that  $J^{h,K}(\cdot)$ is
bounded  on $\mcal(\CP^1)$,   and thus $I^{h,K}(\cdot)$ is
well-defined. Further, we have the following.
\begin{lem}
    \label{prop-080209a}
    The function $\tilde I^{h,K}(\cdot)$ on $\mcal(\CP^1)$
    is a rate function.
\end{lem}\begin{proof} By Lemma \ref{lem-ue} and
   Lemma \ref{LSCJG}, the function
    $$ -\frac12 \ecal_h(\cdot)+J^{h,K}(\cdot)$$
is well defined  on $\mcal(\CP^1)$, bounded below, and
lower semi-continuous.
This implies the claim.
\end{proof}

Next we prove convexity of the
rate function. Convexity of the (unweighted) logarithmic energy is well-known
  \cite{ST,BG,BB}.

\begin{lem} \label{CONVEX} $I^{h, K}$ is a strictly  convex  function on $\mcal$ (with possible values $+ \infty$) \end{lem}

  \begin{proof} Convexity of the energy functional is proved in
  Proposition \ref{CONVEXITY}.
To complete the proof, we note that the  `potential term' $J^{h,
K} (\mu)  = \sup_K  U^{\mu}_{h}(z) $ is a maximum of affine
functions of $\mu$, hence is convex.

\end{proof}

\subsection{\label{GLOBALGENUS0} The global minimizer of $I^{h, K}$}

Since $I^{h, K}$ is lower semi-continuous it has a minimum on
$\mcal(\CP^1)$ and also on each closed  ball $B(\sigma, \delta)
\subset \mcal(\CP^1)$. In this section we show that the global
minimum is the equilibrium measure $\nu_{K,h}$ for the data
$(h, \nu)$
 defining $I^{h, K}$.

The equilibrium measure $\nu_{K, h}$ is the unique maximizer
of $\ecal_{h}(\mu) $ on $\mcal(K)$. Our function differs
from this constrained function in not being constrained to
$\mcal(K)$ but rather   possessing the term $ \sup_K
U^{\mu}_{h}$. We need to show that this term  behaves like a
`Lagrange multiplier' enforcing the constraint. Unfortunately, it
is not `smooth' as a function of $\mu$, so we
cannot use
calculus alone to demonstrate this.

%
%
%

\begin{lem}\label{GLOBALMAX} The global minimizer of $I^{h, K}$ is $\nu_{K, h}$. The global minimum
is $ \frac12  \log \mbox{Cap}_{\omega_h} (K)$.  \end{lem}

 \begin{proof}
\begin{eqnarray} \label{METRICI2a}
2 I^{h, K} (\mu) & = & -  \ecal_h(\mu) +  2\sup_K U^{\mu}_h
\\   &=& -  \int_{\CP^1} (U_{h}^{\mu} - \sup_K U_{h}^{\mu})
 (dd^c U_{h}^{\mu} + \omega_h)  + \sup_K U_{h}^{\mu}.
 \nonumber
\end{eqnarray}

We claim that $ \int_{\CP^1} (U_{h}^{\mu} - \sup_K U_{h}^{\mu})
(dd^c U_{h}^{\mu} + \omega_h) - \sup_K U_{h}^{\mu}$ is maximized
by
$V^*_{K, h}$. By definition, $U^{\mu}_{h} - \sup_{z \in K}
 U^{\mu}_{h}
 \leq 0$ on $K$.
Hence, $U^{\mu}_{h} - \sup_{z \in K} U^{\mu}_{h} \leq
 V_{K, h}^*$. Since  $dd^c U_{h}^{\mu} + \omega_h$ is
 the
 positive measure $d\mu$,
 \begin{eqnarray*}
- 2 I^{h, K} (\mu) &=&   \int_{\C} (U_{h}^{\mu} - \sup_K
U_{h}^{\mu})
 (dd^c U_{h}^{\mu} + \omega_h)  - \sup_K U_{h}^{\mu}
  \\ && \\ &\leq & \int  V_{K, h}^*
 (dd^c U_{h}^{\mu}   + \omega_h) - \sup_K U_{h}^{\mu} \\
  & = & \int (U_{h}^{\mu}   - \sup_{z \in K}
U^{\mu}_{h} )  dd^c   V_{K, h}^* - \sup_K U_{h}^{\mu}
+ \int_{\CP^1}V_{K, h}^* \omega_h
   \\
  & = & \int (U_{h}^{\mu}   - \sup_{z \in K} U^{\mu}_{h} )
  (dd^c   V_{K, h}^* + \omega_h)  + \int_{\CP^1}V_{K, h}^* \omega_h    \\
  & \leq &
    \int V_{K, h}^*   (dd^c V_{K, h}^* + \omega_h) +
    \int_{\CP^1}V_{K, h}^* \omega_h\\
    &= &
F_{K, \omega_h}  . \end{eqnarray*}
In the third line, we integrated $dd^c$ by parts, and used that constants
integrate to $0$ against $dd^c V^*_{K,h}$.
In the next to last
line,
we again use that $U_h^{\mu} - \sup_{z \in K} U_h^{\mu} \leq
 V_{K, h}^*$. In the  last equality, we used
Proposition \ref{WvsU}.

Since  $$
\int_{\CP^1} (U_{h}^{\nu_{h, K}} - \sup_K U_{h}^{\nu_{h, K}} )
(dd^c U_{h}^{\nu_{h, K}} + \omega_h) - \sup_K U_{h}^{\nu_{h, K}} =
\int V_{K, h}^*   (dd^c V_{K, h}^* + \omega_h) + F_{K, \omega_h},
$$ we see that $I^{h, K}$ is  is minimized  by $\nu_{h, K}$.

One easily checks that all of the inequalities are equalities for
$\nu_{ K,h}$. When $K$ is regular, $V_{K, h}^* = (U_{h}^{\nu_{K, h}}
- \sup_K U_{h}^{\nu_{K, h}})$ is continuous. We can determine the
sup by integrating both sides against $\omega_h$ as in Proposition
\ref{WvsU}:
$$\int V_{K, h}^* \omega_h = - \sup_K U_{h}^{\nu_{K, h}}. $$

We have,
\begin{eqnarray*}
- 2 I^{h, K} (\nu_{K, h}) &=&   \int_{\CP^1} (U_{h}^{\nu_{K, h}} -
\sup_K U_{h}^{\nu_{K, h}})
 (dd^c U_{h}^{\nu_{K, h}} + \omega_h)  - \sup_K U_{h}^{\nu_{K, h}}    \\
  & = &
    \int V_{K, h}^*   (dd^c V_{K, h}^* + \omega_h) +
    \int_{\CP^1}V_{K, h}^* \omega_h
    = 
F_{K, \omega_h}  . \end{eqnarray*}

On the other hand, since $U_{\omega_h}^{\nu_{K,h}}=-F_{K,\omega_h}$ on $K$,
we have that
$$ \log \mbox{Cap}_{\omega_h}(K)=\ecal_{\omega_h}(\nu_{K,h})=
\int U_{\omega_h}^{\nu_{K,h}} d\nu_{K,h}=-F_{K,\omega}.$$
This completes the proof.
\end{proof}

\section{\label{LDP} Large deviations theorems in genus zero: Proof of Theorem \ref{g0}}

In this section, we prove Theorem \ref{g0}. We already know that
$\tilde
I^{h, K}$ is a good rate function. We still need to
prove that it actually is the rate function of the large
deviations principle.
 As in
\cite{BG} (Section 3), it is equivalent to prove that
$$- I(\sigma):= \limsup_{\delta \to 0} \limsup_{N \to \infty}
\frac{1}{N^2} \log {\bf Prob} _N(B(\sigma, \delta)) =
\liminf_{\delta \to 0} \liminf_{N \to \infty} \frac{1}{N^2} \log
{\bf Prob}_N(B(\sigma, \delta)). $$ See Theorem 4.1.11 of
\cite{DZ}.

The proof follows the approach in
\cite{BG,BZ} of
large deviations principles for empirical measures of eigenvalues
of certain random matrices. However, we take full advantage
of the compactness of $\mcal(\CP^1)$, and of the properties of
the Green function $G$, see Lemma
\ref{lem-ue} and Proposition \ref{LSCJG}.

We first prove the result without taking the normalizing constants
$Z_N(h)$ of Proposition \ref{FSVOLZETA2intro} into account. Then
in \S \ref{SLATERG0}, we determine the logarithmic asymptotics of
the normalizing constants.

 \subsection{Heuristic derivation of $I^{h, K}$}

Since the proof of the large deviations
principle
is technical, we
first give a formal or heuristic derivation of the rate function
(\ref{IHKLOCAL}) from   the expression for the joint probability
distribution of zeros in Lemma \ref{APPROXRATE}  in the spirit of
the discussion in \S \ref{sec-lowbrow}. We then fill in the
technical gaps to give a rigorous proof.

\subsubsection{Heuristic derivation}

 Lemma \ref{APPROXRATE} expresses $\vec K_n^N(\zeta_1, \dots,
 \zeta_N)$ as a product of three factors: the normalizing constant
 $\frac{1}{\hat{Z}_N(h)}$, the factor $e^{- N^2 \left( -\frac{1}{2} \ecal^h_N(\mu_{\zeta}) + \frac{N+1}{N}  J_N^{h, \nu}
 (\mu_{\zeta})\right)}$
 and the integration measure $\prod_{j = 1}^N e^{-
2 N \phi(\zeta_j)} d^2 \zeta_j $. The normalizing constant will be
worked out asymptotically in \S \ref{SLATERG0} using the fact that
$\vec K_n^N(\zeta_1, \dots,
 \zeta_N)$ is a probability measure. The integration measure $\prod_{j = 1}^N e^{-
2 N \phi(\zeta_j)} d^2 \zeta_j $ is an invariantly defined smooth
$(N, N)$ on $(\CP^1)^N$ of finite mass independent of $N$, and
thus does not contribute to the logarithmic asymptotics. Hence,
only the factor $\left( -\frac{1}{2} \ecal^h_N(\mu_{\zeta}) +
\frac{N+1}{N}  J_N^{h, \nu}
 (\mu_{\zeta})\right)$ contributes to the rate function.

 The term $\ecal^h_N(\mu_{\zeta})$ closely resembles the energy
 except that the diagonal $D$ has been punctured out of the
 domain of integration, as it must since $\mu_{\zeta}$ has
 infinite energy. It must be shown that the true energy is the correct limiting
 form when measuring log
 probabilities of balls of measures.

The second term satisfies,
\begin{equation} \lim_{N \to \infty} J_N^{h, \nu} (\mu_{\zeta})  = \log ||
e^{U^{\mu}_h} ||_{L^N(\nu)} \uparrow \log || e^{U^{\mu}_h}
||_{L^{\infty}(\nu)} = \sup_K U^{\mu}_h
\end{equation}
monotonically as $N \to \infty$. Thus,  it is natural to
conjecture that the rate function for large deviations of
empirical measures is given by (\ref{IHKLOCAL}).

We now turn to the rigorous proof.

\subsection{Proof of the upper bound}
In this section,
we prove the upper bound part of the large deviation
principle, that is we prove that
\begin{equation} \label{UBTOPROVE}
    \lim_{\delta \downarrow 0} \limsup_N \frac{1}{N^2}  \log \; {\bf
Prob}_N (B(\sigma, \delta))
\leq - \tilde I^{h,K}(\sigma).
\end{equation}

The first step is:
\begin{lem}\label{BMUB}  Fix $\epsilon>0$. If $\nu$ satisfies
    the Bernstein-Markov condition (\ref{BM}),
then there exists a $N_0=N_0(\epsilon)$ such that for all
$N>N_0$ and all $\mu_\zeta\in \mcal(\CP^1)$,
$$ \log ||e^{U^{\mu_\zeta}_h}||_{L^N(\nu) }
\geq \sup_{z \in K} U^{\mu_\zeta}_h - \epsilon \,.$$
\end{lem}
\begin{proof}
We are assuming that, for  all $s \in H^0(C, L^N)$,
\begin{equation} \sup_{z \in K} |s(z)|_{h^N} \leq C_{\epsilon}
e^{\epsilon N} \left(\int_K |s(z)|^2_{h^N} d \nu(z) \right)^{1/2}.
\end{equation}
By Lemma \ref{NORMSUPHI} we may write
$$|s_\zeta(z)|^2_{h^N} =
e^{N U_h^{\mu_\zeta}(z)} e^{  N (\int \phi d\mu_{\zeta} - E(h))}
$$ Hence,  \begin{equation}
    \begin{array}{lll} ||e^{U_h^{\mu_\zeta}}||_{L^N} e^{   (\int \phi d\mu_{\zeta} - E(h))}& = &
        \left(\int_K |s_\zeta(z)|^2_{h^N} d \nu(z)
\right)^{1/N}\\ && \\
& \geq & \left( C^{-1}_{\epsilon} e^{- N \epsilon} \sup_{z \in K}
|s_\zeta(z)|_{h^N}^2  \right)^{\frac{1}{N}} \\ && \\
& \implies & \log ||e^{U_h^{\mu_\zeta}}||_{L^N} \geq \sup_{z \in K} U_h^{\mu_\zeta} - \epsilon
+ \frac1N\log C_\epsilon, \end{array} \end{equation} for all $\epsilon > 0$.
\end{proof}

Write
$$\Theta_N=-\frac{1}{N^2} \log\hat{Z}_N(h)\,.$$
As we will see in the course of the proof of Lemma
\ref{NC}, $\Theta_N\to_{N\to\infty} \log Cap_{\omega_h}(K)$.

By  Lemma \ref{JPDINV} and Lemma \ref{APPROXRATE},
\begin{equation}
    \label{eq-180209b}
    \frac{1}{N^2}  \log \; {\bf Prob}_N (B(\sigma, \delta))
= \frac{1}{N^2} \log \int_{\zeta \in (\CP^1)^N : \mu_{\zeta} \in
B(\sigma, \delta)} e^{- N^2 I_N(\mu_{\zeta}) } \kappa_N
+\Theta_N.
\end{equation}
Fix $M\in \R$ and let
$G_h^M=G_h\vee (-M)$ be the truncated Green function.
By Lemma \ref{lem-ue}, $G_h^M$ is continuous on $\CP^1\times \CP^1$.
Further, with notation as in \eqref{eq-180209a},
\begin{eqnarray*}
    -\frac{1}{N^2}\sum_{i<j}G_h(\zeta_i,\zeta_j)&\geq&
-\frac{1}{N^2}\sum_{i<j}G_h^M(\zeta_i,\zeta_j)\\
&\geq&
-\frac12 \int\int_{\CP^1\times\CP^1}
G_h^M(z,w)d\mu_\xi(z) d\mu_\xi(w)-\frac{C(M)}{N}=
\ecal_h^M(\mu_\xi)
-\frac{C(M)}{N}
,
\end{eqnarray*}
where the constant $C(M)$ does not depend on $\xi$. Using Lemma
\ref{BMUB} (and also Corollary \ref{UUB}), we then have that for
any $\epsilon>0$ and all $N>N_0(\epsilon)$,
\begin{equation}\begin{array}{lll}
 \frac{1}{N^2}  \log \; {\bf Prob}_N (B(\sigma, \delta))
 &\leq & \frac{1}{N^2}\log\int_{\xi\in (\CP^1)^N: \mu_\xi\in B(\sigma,\delta)}
 e^{\frac{N^2}{2}\ecal_h^M(\mu_\xi)-N^2J_h^K(\mu_\xi)}
\kappa_N
\\ && \\ &+&
  (
  \Theta_N
 +\frac{C'(M)}{N}+\epsilon),
\end{array} \end{equation}
 for some constant $C'(M)$. It follows that
$$   \limsup_N \frac{1}{N^2}  \log \; {\bf
Prob}_N (B(\sigma, \delta))\leq
\Theta_N+ \limsup_{\delta \downarrow
0} \sup_{\mu \in B(\sigma, \delta)} - \left( -
\frac12\ecal_h^M(\sigma) + J_h^K(\sigma)\right)\,.$$
 Here, we use that
\begin{equation} \label{ZEROa} \frac{1}{N^2}\log \int_{(\CP^1)^N}
\kappa_N=
 O(\frac{\log N}{N}),
\end{equation}
which follows from Lemma \ref{INVAR}.

 It then follows from the
continuity of $\ecal_h^M(\sigma)$ and the lower semi-continuity of
$J_h^K(\sigma)$ (see Lemma \ref{LSCJG}) that
$$  \lim_{\delta \downarrow 0} \limsup_N \frac{1}{N^2}  \log \; {\bf
Prob}_N (B(\sigma, \delta))\leq
\lim_{N\to\infty} \Theta_N+
\frac12\ecal_h^M(\sigma)-J_h^K(\sigma) + \epsilon\,.$$ Since
$\ecal_h^M(\sigma) \to \ecal_h(\sigma)$ as $M \to \infty$  by
monotone convergence, and since $\epsilon$ is arbitrary,  we
obtain \eqref{UBTOPROVE}. \qed

\subsection{Proof of the lower bound}
In this section, we prove that
\begin{equation} \label{LBTOPROVE}
    \lim_{\delta \downarrow 0} \liminf_N \frac{1}{N^2}  \log \; {\bf
Prob}_N (B(\sigma, \delta))
= - I(\sigma).
\end{equation}
The strategy is again similar to \cite{BG} and \cite{BZ}. Note first that
to prove \eqref{LBTOPROVE},
it is enough to find, for any $\sigma$ with $I(\sigma)<\infty$,
a sequence $\sigma_\epsilon\to_{\epsilon\to 0}
\sigma$ weakly
in $\mcal(\CP^1)$ such that $I(\sigma_\epsilon)\to I(\sigma)$
and \eqref{LBTOPROVE} holds for $\sigma_\epsilon$.

So, define $\sigma_\epsilon=e^{\epsilon \Delta_\omega}\sigma$ as in
Lemma \ref{REGLEMMA} of the appendix. By
 that lemma, the property $I(\sigma_\epsilon)\to_{\epsilon\to
 0}I(\sigma)$ holds. It thus only remains to prove \eqref{LBTOPROVE}
 when $\sigma$ is replaced by $\sigma_\epsilon$. Thus, the
 large deviations lower bound is a consequence of the following lemma.
 \begin{lem}
     \label{lbsmoothed}
     Let $\sigma=f\omega\in \mcal(\CP^1)$ with $f$ a strictly
     positive and continuous function on $\CP^1$. Then,
     \eqref{LBTOPROVE} holds.
 \end{lem}
 \begin{proof}
We follow \cite{BZ}, Lemma 2.5.
It will be convenient to consider three charts on $\CP^1$, denoted
$W_0, W_1,W_2$, with distance $d_{W_i}$ on $W_i$,
that cover $\CP^1$, and a constant $R$,
in such a way that for any two points $z,w\in \CP^1$ there exists
a chart $W_{z,w}$ with distance
$d_{W_{z,w}}$ so that in local coordinates, $d(z,w):=d_{W_{z,w}}(z,w)\leq R$
(If more than one such chart exists for a given pair $z,w$, fix one
arbitrarily as $W_{z,w}$; The charts $W_0$ can be taken as the standard
chart $\C$, with $W_1$ and $W_2$ its translation to two fixed
distinct points in $\CP^1$.)

Construct a sequence of discrete probability measures
$$d\sigma_N = \frac{1}{N} \sum_{j = 1}^N \delta_{Z_j} \in B(\sigma, \delta)$$
with the following properties:
\begin{enumerate}
\item $\sigma_N\in B(\sigma,\delta/2) $ for all $N$ large;
\item $d(Z_i, Z_j) \geq \frac{C(\sigma,\delta)}{\sqrt{N}}. $
\end{enumerate}
(Since in a local chart, $\sigma$ possesses a bounded
density with respect to Lebesgue's measure,
such a sequence can be constructing by adapting to the local
charts $W_i$ the construction in the proof of Lemma 5 in
\cite{BZ}.) Define
$$D_I^{\eta} = \{\zeta \in (\CP^1)^N: d(\zeta_j,  Z_j) \leq
\frac{\eta}{N}, \;\; j = 1, \dots, N \}. $$ Then, for $\eta$ small
enough and all $N$ large, all $\zeta \in D_I^{\eta}$ satisfy that
$\mu_{\zeta} \in B(\sigma, \delta).$ Since $D_I^{\eta} \subset
B(\sigma, \delta),$
\begin{equation} \label{eq-180209c}
    {\bf
    Prob}_N (B(\sigma, \delta))  \geq \int_{D_I^{\eta}}  e^{- N^2
I_N(\mu_{\zeta}) }  \kappa_N+\Theta_N.
\end{equation}
(Recall Lemma \ref{JPDINV} for the definition of the $(N,N)$ form
$\kappa_N$.)
 By Proposition \ref{prop-grub} and our
construction, there exists a constant $C_1=C_1(\eta,\sigma)$ with
$C_1\to_{\eta\to 0} 0$ such that for any $\xi\in D_I^\eta$ and
$i,j\in \{1,\ldots,N\}$, $i\neq j$,
$$G_h(\xi_i,\xi_j)\geq G_h(Z_i,Z_j)-C_1(\eta)\,.$$
For $\epsilon>0$, set
$$\ecal_h^\epsilon(\sigma):=
\int_{\CP^1\times \CP^1\setminus D_\delta}
G_h(z,w) d\sigma(z)d\sigma(w)\,,$$
where
$D_\epsilon=\{(z,w)\in (\CP^1\times \CP^1): d(z,w)<\epsilon\}$.
We have by monotone convergence that
$$\ecal_h(\sigma)=
\lim_{\epsilon\to 0}
\ecal_h^\epsilon(\sigma).$$
Because $G_h$ is continuous on $\CP^1\times \CP^1\setminus D_\epsilon$,
we have that
$$N^{-2}\sum_{i\neq j, d(Z_i,Z_j)\geq\epsilon} G_h(Z_i,Z_j)\geq
\ecal_h(\sigma)-C_2(\epsilon,\delta)\,,$$
where for fixed $\epsilon$, $C_2(\epsilon,\delta)\to_{\delta\to 0} 0$.
On the other hand, let
$J_i(r)=\#\{j\in \{1,\ldots,N\}: j\neq i, d(Z_i,Z_j)\in
[r/\sqrt{N},(r+1)/\sqrt{N}]\}$.
From our construction, there exists a constant $C_3=C_3(\sigma,\delta)$
such that $J_i(r)\leq C_3 r$ for any $i\in \{1,\ldots,N\}$ and
all $r<\epsilon \sqrt{N}$, if $\epsilon$ is smaller than some $\epsilon_0$
independent of $\delta$. Thus, applying
Proposition
\ref{prop-grub},
we have that
$$N^{-2}\sum_{i\neq j, d(Z_i,Z_j)\leq\epsilon} |G_h(Z_i,Z_j)|\leq
\frac{C_3(\epsilon,\delta,\sigma)}{N^2}
\sum_{i=1}^N \sum_{k=1}^{\epsilon \sqrt{N}}
k \log(k/\sqrt{N}) =
O(\frac{\log N}{\sqrt{N}})\,.$$
For $\epsilon'>0$ given,
fix $\epsilon>0$ so that
$$|\ecal_h(\sigma)-
\ecal_h^\epsilon(\sigma)|<\epsilon'\,.$$
Then, taking $N\to\infty$
we conclude that
$$\limsup_{N\to\infty}
|N^{-2}\sum_{i\neq j} G_h(Z_i,Z_j)-
\ecal_h(\sigma)|\leq C_2(\epsilon,\delta)+\epsilon'\,.$$
In particular, taking $\delta=\delta(\epsilon')$ small enough gives
$$\limsup_{N\to\infty}
|N^{-2}\sum_{i\neq j} G_h(Z_i,Z_j)-
\ecal_h(\sigma)|\leq 2\epsilon'\,.$$
By Proposition \ref{LSCJG}, reducing $\delta$ further if necessary, we
also have $|J^{h,K}(\sigma_N)-
J^{h,K}(\sigma)|\leq \epsilon'$.
Combining these estimates and substituting in \eqref{eq-180209c},
one gets that for any $\epsilon'>0$ and all $N$ large enough,
\begin{equation}
    \label{ofer-new}
    {\bf  Prob}_N (B(\sigma, \delta))  \geq
e^{-N^2I(\sigma)-3\epsilon' N^2} \int_{D_I^{\eta}}  \kappa_N
\end{equation}

To complete the proof, we again use (\ref{ZEROa}). Indeed, as
above (see Proposition \ref{INVAR}), $\kappa$
is a smooth positive $(1,1)$ form on
 $\CP^1$. Now, $D_I^{\eta} \subset (\CP^1)^N$ is a product of the
 one-complex dimensional sets
$D_j^{\eta; N} : = \{\zeta_j: d(\zeta_j,  Z_j) \leq
\frac{\eta}{N}\} \subset \CP^1$. Hence, $$ \int_{D_I^{\eta}}  \kappa_N
=\left(\int_{D_1^{\eta; N}} \kappa
 \right)^N
. $$
 Since $\kappa$
 is a smooth area form,
 $\int_{D_1^{\eta; N}}  \kappa
 \sim C N^{-2}. $ It follows
 that
 $$ \frac{1}{N^2} \log \int_{D_I^{\eta}}  \kappa_N
= O ( \frac{\log N}{N}).$$
Combined with \eqref{ofer-new} and the fact that $\epsilon'$ was arbitrary,
this completes the proof.
\end{proof}

\subsection{\label{SLATERG0} The normalizing constant: Proof of Lemma \ref{NC}}
Finally, we consider the normalizing constants of Proposition
\ref{FSVOLZETA2intro},
 in particular the   determinant $\det \acal_N(h, \nu) = \det \left(
\langle z^k, \psi_{\ell} \rangle \right)$ of the change of basis
matrix (\ref{ACALH}).  Here, $\langle \cdot, \cdot\rangle$ is the
inner product $G_N(h, \nu)$. The same asymptotics have bee studied
before in the theory of orthogonal polynomials (see e.g. \cite{B})
and in the setting of line bundles in   \cite{Don2,BB}. The
following gives an alternative to the proof in  \cite{BB} in the
special case at hand.

 We claim that $$\lim_{N \to \infty}
\frac{1}{N^2} \log \hat{Z}_N(h) = \frac{1}{2} \log
\mbox{Cap}_{\omega_h} (K).
$$

\begin{proof}
We prove this by combining the  large deviations result for the
un-normalized probability measure with the fact that $\PR_N$ is a
probability measure. By Lemma \ref{GLOBALMAX} and proof of the
large deviations upper bound,

\begin{equation} \begin{array}{lll}  0 & = & \lim_{N \to \infty} \frac{1}{N^2} \log {\bf Prob}
_N(\mcal(\CP^1) ) \\ && \\ & \leq& \limsup_{N \to \infty}
\frac{-1}{N^2} \log \hat{Z}_N(h)   - \inf_{\mu \in \mcal(\CP^1)}
I^{h, K} (\mu) \\ && \\
& = & \limsup_{N \to \infty} \frac{-1}{N^2} \log \hat{Z}_N(h)   -
I^{h, K}(\nu_{h, K})  \\ && \\
& = & \limsup_{N \to \infty} \frac{-1}{N^2} \log \hat{Z}_N(h)
-\frac{1}{2} \log \mbox{Cap}_{\omega_h} (K).
\end{array}
\end{equation}
A similar argument using the large deviations lower bound shows
the reverse inequality for $\liminf_{N\to\infty} \frac{-1}{N^2} \log
\hat{Z}_N.$

\end{proof}

\begin{cor}
$$  \lim_{N \to \infty}
\frac{1}{N^2}  \log |\det \acal_N(h)|^{-2} = -  \frac{1}{2} E(h) +
\frac12 \log \mbox{Cap}_{\omega_h} (K).
$$
\end{cor}

\begin{proof} By Lemma \ref{JPDINV},
 $$\lim_{N \to \infty} \frac{1}{N^2} \log \hat{Z}_N(h) =   \lim_{N \to \infty}
\frac{1}{N^2}  \log |\det \acal_N(h)|^{-2} + \frac{1}{2} E(h). $$
\end{proof}

\section{\label{APPENDIX} Appendix}

This Appendix contains proofs of some technicalities used in the
proofs of Theorem \ref{g0}.

\subsection{\label{REG} Regularization of measures}

In the large deviations lower bound, we will need to prove that
any $\mu \in \mcal(\CP^1)$ may be weakly approximated by measures
with continuous densities. In \cite{BZ}, this was proved using
convolution with Gaussians; since we are working on $\CP^1$, we
need a suitable replacement.

We once again use the auxiliary \kahler metric and its Laplacian
$\Delta_{\omega}$. It generates the heat operator $e^{t
\Delta_{\omega}}$. We denote its heat kernel by \begin{equation}
\label{HEAT} K_{\omega}(t, z,w) = \sum_{j = 1}^{\infty} e^{-t
\lambda_j} \phi_j(z) \phi_j(w),  \end{equation} and write
\begin{equation} e^{t \Delta_{\omega}}\mu (z) = \int_{\CP^1} K_{\omega}(t, z, w)d \mu(w). \end{equation}
It is well-known and easy to see that $e^{t \Delta_{\omega}}\mu
(z) \in C^{\infty}(\CP^1)$ for any $\mu \in \mcal(\CP^1)$.  The
following simple Lemma is sufficient for our purposes:
\begin{lem} \label{REGLEMMA} If $I^{h, K} (\mu) < \infty$, then $e^{t \Delta_{\omega}} \mu \to
\mu$ in $\ecal_+$ as $t \to 0^+$. In particular, $I^{h, K} (e^{t
\Delta} \mu) \to I^{h, K}(\mu). $ Moreover, $e^{t \Delta_{\omega}}
\mu \in \mcal(\CP^1).$
   \end{lem}

\begin{proof} It is well known (and follows by the maximum
principle for the heat equation) that $K_{\omega}(t, z, w) > 0$.
Hence, $\mu_t : = \left( e^{t \Delta_{\omega}} \mu\right)\omega $
is a positive measure. Further, $\int_{\CP^1} \mu_t = 1$ since
$\int_{\CP^1} K(t, z, w) \omega = 1$.

 As observed above, it is equivalent to say that
$e^{t \Delta_{\omega}} \mu \to \mu$ in $H^{-1}(\CP^1)$. It
suffices to observe that by monotone convergence,
$$ \lim_{t \to 0^+} \sum_{j = 1}^{\infty}\frac{1 - e^{t
\lambda_j}}{\lambda_j}  |\mu(\phi_j)|^2 = 0. $$
\end{proof}

 \subsection{\label{IP} Residue at infinity of certain integrals}  This section is
 rather technical. We go over the `residue at infinity' of a
 number of integrals that arise in the calculations of the LDP.

 Integrals of the form $\int_{\C} v dd^c
 u$  often arise when $u, v$ are
potentials for probability measures.
 The subtlety is the `boundary term' at infinity. To illustrate, we note that although $dd^c \log (1 + |z|^2)^{/12}$ is exact,
 the integral (e.g.) $\int_{C} dd^c \log (1 +
 |z|^2)^{1/2} = 1$ and not zero.  The integral has been studied in all dimensions in \cite{BT}, and
although we are only
 working in dimension one we will follow their presentation.

 The  formula for $\int_{\C} v dd^c u$  depends on the class of functions that $u, v$ belong
 to. Following \cite{BT}, one defines the Robin function
  $\rho_u$ of a subharmonic function $u$ by
 \begin{equation}\label{ROBINDEF}  \rho_u(z) = \limsup_{\lambda \to \infty} \left(u
 (\lambda z) - \log^+ |\lambda z| \right). \end{equation}
Here,  $u^*$ denotes the upper semi-continuous
 regularization of $u$. We denote by  $\lcal$  the Lelong class of subharmonic
 functions satisfying $u(z) \leq \log^+ |z| + C$, and put    $\lcal_{\rho} := \{u \in \lcal(\C): \rho_u \not= -
 \infty\}. $ Note that $\rho_1 \equiv - \infty$.

\subsubsection{Case (i): $u, v \in \lcal_{\rho}$ }

 It is proved in \cite{BT} that if $u, v \in
\lcal_{\rho}$, then
 \begin{equation} \label{BT} \int_{\C} u dd^c v - v dd^c u = 2 \pi
 (\rho_u^*(\infty) - \rho_v^*(\infty)), \end{equation}

The formula applies in particular to:

 \begin{enumerate}

 \item $\int_{\C} \psi dd^c \psi$ where $\omega = dd^c \psi$ on
 $\C$; here $\omega$ is the curvature $(1,1)$ form of any
 Hermitian metric on $\ocal(1)$. For any such metric, $\psi \in
 \lcal_{\rho}$.

 \item $\int_{\C} \log |z - w| dd^c \psi$ where $\psi$ is as
 above.

 \item $\int_{\CP^1} \log ||s_{\zeta}||^2_{h^N} \omega_h$, where as usual $s_{\zeta} = \prod_{j = 1}^N (z - \zeta_j) e^N. $

 \end{enumerate}

 In the case (1),  the integration by parts formula just reproduces the same
 expression, but we include it for future reference.

In case (2), we first consider the Fubini-Study case, where the
Hermitian metric $h_{FS}$ is locally given in the standard frame
by  $\psi(w) = \log (1 + |w|^2)^{1/2}$. It is simple to see that
 \begin{equation} \label{FSINT0} \begin{array}{lll} \int_{\C} \log |z - w| dd^c \log (1 +
 |w|^2)
 & = & \frac{1}{2} \log (1 +
 |z|^2) . \end{array} \end{equation}
 This follows from (\ref{FUNDSOL}) and the fact that  $\rho_{\log (1 + |w|^2)}(\infty) = \rho_{\log |z - w|} (\infty) =0$,
 since   for sufficiently large $|w|$,
 $$\begin{array}{l} \log |z - \lambda w| - \log |\lambda w| = \log |1 - \frac{z}{\lambda w}| = \Re \log (1
 - \frac{z}{\lambda w}) = - \Re \frac{z}{\lambda w} + \dots = o(1),\\ \\
  \log (1 + |\lambda w|^2)^{1/2} - \log |\lambda w| = \log (1 + |\lambda w|^{-1})^{1/2}= o(1).\end{array} $$

 In the case of a general Hermitian metric on $\ocal(1)$, we may
 write $h = h_{FS} e^{- \Phi}$ where $\Phi \in C^{\infty}(\CP^1)$.
 Then on $\C$, the weight has the form $\log (1 + |w|^2)^{1/2} +
 \Phi \in \lcal_{\rho}$. Now $\rho_{\psi} (\infty) =
 \Phi(\infty)$. Hence,
\begin{equation} \label{FSINT} \begin{array}{lll} \int_{\C} \log |z - w| dd^c
  \psi
 & = &   (\frac{1}{2} \psi(z) - 2\pi
 \rho_{\psi}(\infty)) =  \frac{1}{2} \psi(z) - 2\pi \Phi(\infty).
 \end{array} \end{equation}

In case (3), we again express the  Hermitian metric as $h = e^{-
\Phi} h_{FS}$, and then
\begin{equation} \label{LOGSINTa} \begin{array}{lll}  \int_{\CP^1} \log ||s_{\zeta}||^2_{h^N} \omega_h & = & 2 \sum_{j  = 1}^N
\int_{\C} \log |z - \zeta_j| dd^c (\log (1 + |z|^2) +
\Phi(z)) \\ && \\
& = &  \sum_j \log (1 + |\zeta_j|^2) + 2 \pi (\Phi(\zeta_j)) -
\Phi(\infty))\\ \\
& = &  N \int (\log (1 + |w|^2) + 2\pi (  \Phi(w) - \Phi(\infty))
d\mu_{\zeta}(w)) .
\end{array}
\end{equation}
This integral is also evaluated in Lemma \ref{NORMSUPHI}.

 \subsubsection{Case (ii) $u, v \in C^2(\CP^1)$ }
 Obviously, in this case we can simply integrate by parts.

\subsubsection{Case (iii) $u \in \lcal_{\rho}, v \in C^2(\CP^1)$}

We are interested in three cases:

\begin{enumerate}

 \item  $\int_{\C} \log |z - w| dd^c v$ where $v \in
C^2(\CP^1)$.

\item  $\int_{\C} \log |z - w| dd^c \phi$ where $e^{- \phi}$ is
the local expression on $\C$ of a global smooth Hermitian metric
$h_{\phi}$  on $\ocal(1)$. Equivalently,
 there exists a global continuous $(1,1)$ form
$\omega$  such that $\int_{\CP^1} \omega = 1$ and such that
$\omega - \delta_{\infty} = dd^c u$.

\item The integral $\int_{\C} u dd^c v$, where $u$ is the
potential on $\C$ of a global smooth  hermitian metric on
$\ocal(1)$.

\item $v \equiv 1, i.e. \; \int_{\C} dd^c \phi = \int_{\CP^1}
\omega = 1. $

\end{enumerate}

In case (1), the logarithmic factor lies in $\lcal_{\rho}$ but $v
\notin \lcal_{\rho}$.  When $v \in C^2(\CP^1)$, we claim that
\begin{equation} \label{LOGINTb} \int_{\C} \log
|z - w| dd^c v =  \frac{1}{2} v(z) - 2\pi v(\infty).
\end{equation} Indeed, we   form the metric $h_{FS} e^{- v}$, apply
(\ref{FSINT}) with $\psi = \log (1 + |w|^2)^{1/2} + v$ and then
subtract the integral for $\log (1 + |w|^2)^{1/2}$.

Similarly, in  case (2) we write the Hermitian metric $h_{\phi} =
e^{- \phi}$ in the form $h_{\phi} = e^{- \psi} h_{FS}$ with $\psi
\in C^{\infty}(\CP^1)$. Then by (\ref{FSINT}) and (\ref{LOGINTb}),
\begin{equation} \label{LOGROBIN} \begin{array}{lll} \int_{\C} \log |z - w| dd^c \phi & =
&\int_{\C} \log |z - w| dd^c \log (1 + |w|^2)^{1/2} + \int_{\C}
\log |z - w| dd^c \psi
\\ && \\
&= &  \frac{1}{2}\log (1 + |z|^2)^{1/2} + \frac{1}{2}
\psi(z) - 2\pi \psi(\infty)\\ && \\
& = & \phi(z)/2 - 2 \pi \rho_{\phi}(\infty). \end{array}
\end{equation}

By a similar calculation, in case  (3) we have,
\begin{equation} \label{uROBIN} \int_{\C}   u dd^c v = \int_{\C} v dd^c u +
2 \pi (\rho_u(\infty) -  v(\infty)), \end{equation} while in  case
(4) we have \begin{equation} \label{ROBINPHI} \int_{\C} dd^c \phi
= - 2 \pi  \rho_{\phi} (\infty) = - 2 \pi.
\end{equation}

\end{document}